\documentclass[12pt]{article}

\usepackage{amsthm,amsmath,amssymb,amsfonts}
\usepackage{color}
\usepackage{a4wide}
\newcommand{\vphi}{\vcg{\varphi}}
\newcommand{\kappad}{\kappa_{\delta}}

\newcommand{\dx}{{\, \rm d}x}

\newcommand{\Div}{{\rm div}\,}
\newtheorem{thm}{Theorem}
\newtheorem{lem}[thm]{Lemma}
\newtheorem{prop}[thm]{Proposition}
\newtheorem{df}{Definition}
\newtheorem{rmk}{Remark}

\newcommand{\Ov}[1]{\overline{#1}}

\newcommand{\Un}[1]{\underline{#1}}
\newcommand{\bo}{| _{\partial\Omega}}

\newcommand{\vy}{Y_{k}}
\newcommand{\vf}{{\vc F}}

\newcommand{\vS}{{\bf S}}
\newcommand{\vw}{\omega}

\newcommand{\vr}{\varrho}

\newcommand{\vp}{\varphi}

\newcommand{\vrd}{\vr_\delta}
\newcommand{\vcYd}{\vec{Y}_\delta}

\newcommand{\vtd}{\vt_\delta}
\newcommand{\vud}{\vu_\delta}
\newcommand{\vt}{\vartheta}

\newcommand{\vu}{\vc{u}}

\newcommand{\vc}[1]{{\bf #1}}
\newcommand{\vcg}[1]{{\pmb #1}}
\newcommand{\F}[1]{$\mathbb{#1}$}
\newcommand{\Grad}{\nabla}

\newcommand{\tn}[1]{\mbox {\F #1}}

\newcommand{\dS}{{\rm d}S}

\newcommand{\lr}[1]{\left( #1 \right)}
\newcommand{\intO}[1]{\int_{\Omega} #1\dx}
\newcommand{\intpO}[1]{\int_{\partial\Omega} #1 \, {\rm d} {S}}

\newcommand{\intOB}[1]{\int_{\Omega} \left( #1 \right) \ \dx}

\newcommand{\D}{{\vc{D}}}
\newcommand{\ep}{\varepsilon}

\newcommand{\R}{\mathbb{R}}

\newcommand{\eq}[1]{\begin{equation}
\begin{split}
#1
\end{split}
\end{equation}}
\newcommand{\eqh}[1]{\begin{equation*}
\begin{split}
#1
\end{split}
\end{equation*}}

\newcommand{\sumkN}[1]{\sum_{k=1}^n #1}
\newcommand{\sumlN}[1]{\sum_{l=1}^n #1}

\begin{document}

\title{Weak and variational entropy solutions to the system describing steady flow
of a compressible reactive mixture}
\author{Tomasz Piasecki$^1$ and  Milan Pokorn\' y$^2$}
\maketitle

\bigskip

\centerline{ $^1$ Institute of Mathematics, Polish Academy of Sciences}
\centerline{ul. \'Sniadeckich 8, 00-656 Warszawa, Poland}
\centerline{e-mail: {\tt piasecki@impan.pl}}

\centerline{$^{2}$ Charles University, Faculty of Mathematics and Physics}
\centerline{Mathematical Inst. of Charles University, Sokolovsk\' a 83, 186 75 Prague 8, Czech Republic}
\centerline{e-mail: {\tt pokorny@karlin.mff.cuni.cz}}
\vskip0.25cm
\begin{abstract}
We consider a system of partial differential equations which describes steady flow of a compressible heat conducting chemically reacting gaseous mixture. We extend the result from Giovangigli, Pokorn\'y, Zatorska (2015) in the sense that we introduce the variational entropy solution for this model and  prove existence of a weak solution for $\gamma >\frac 43$ and existence of a variational entropy solution for any $\gamma >1$. The proof is based on improved density estimates.
\end{abstract}


\smallskip


\section{Introduction}

Chemically reacting mixtures appear in many real-life situations, especially in chemical engineering (\cite{Ro95}), combustion (\cite{Wi111}), description of some atmospheric phenomena (\cite{SePa97}) and many others. There are many models of mixtures which can be derived from different  general physical models depending on the phenomena which we want to study. We may start from molecular theories like the kinetic theory, statistical mechanics and thermodynamics or from the macroscopic theories like continuum physics and continuum thermodynamics.

Here, we rely on the latter. 
We continue  the program started in \cite{MuPoZa} which was applied to a special situation for the steady problem in \cite{GPZ}.  

More precisely, we investigate a system of partial differential equations describing steady flow of chemically
reactive, heat conducting, gaseous mixture. The system, which composes of the steady compressible Navier--Stokes--Fourier system coupled with the balance of mass fractions, reads 
\begin{equation}\label{1.1}
\begin{array}{c}
\Div (\vr \vu) = 0,\\
\Div (\vr \vu \otimes \vu) - \Div \tn{S} + \Grad \pi =\vr \vc{f},\\
\Div (\vr E\vu )+\Div(\pi\vu) +\Div\bf{Q}- \Div (\tn{S}\vu)=\vr\vc{f}\cdot\vu,\\
\Div (\vr Y_k \vu)+ \Div \vf_{k}  =  m_k\vw_{k},\quad k\in \{1,\ldots,n\}.
\end{array}
\end{equation}
In the above equations $\tn{S}$ denotes the viscous part of the stress tensor, 
$\pi$ the internal pressure of the fluid, 
$\vc{f}$ the external force, $E$ the specific total energy, $\bf{Q}$ the heat flux, $\vw_{k}$ the molar production 
rate of the $k$-th species, $\vf_k$ the diffusion flux of the $k$-th species 
and $m_k$ the molar mass of the $k$-th species which we assume to be equal, hence, without loss of generality 
\begin{equation} \label{eq_mass}
m_1=\ldots=m_n=1.
\end{equation} 
System (\ref{1.1}) is supplemented by the no-slip boundary conditions for the velocity
\begin{equation} \label{1.6}
\vu \bo= \vc{0},
\end{equation}
together with
\begin{equation} \label{1.7}
 \vf_{k}\cdot\vc{n}\bo=0,
\end{equation}
and the Robin boundary condition for the heat flux 
\begin{equation}\label{1.8}
-\vc{Q}\cdot\vc{n}+L(\vt-\vt_{0})=0
\end{equation}
which means that the heat flux through the boundary is proportional to the difference of the temperature 
inside $\Omega$ and the known external temperature $\vt_{0}$.
The coefficient $L$ describes thermal insulation of the boundary and 
for simplicity we assume it to be constant.
We further prescribe the total mass of the mixture 
\begin{equation}\label{conserva}
\intO{\vr}=M>0.
\end{equation}
The mass fractions $Y_k$, $k\in\{1,\ldots,n\}$, are defined by $Y_k=\frac{\vr_k}{\vr}$.
Thus, by definition, they satisfy
\begin{equation}
\sumkN Y_k=1\label{brak}.
\end{equation}
Concerning the chemical production rates, we assume them to be sufficiently regular, bounded 
functions of $\vr,\vt$ and $Y_k$ such that 
$$
\omega_k \geq 0 \quad \textrm{for} \quad Y_k=0.
$$ 
We also assume
\begin{equation} \label{omega_small}
\omega_k \geq - CY_k^r \quad \textrm{for some} \quad C,r>0,
\end{equation}
which means that a species cannot decrease faster 
than proportionally to some positive power of its fraction (a possible natural choice is $r=1$).  
The stress tensor $\tn{S}$ is given by the Newton rheological law as
\begin{equation} \label{vis_ten}
{\tn S} = {\tn S}(\vt, \Grad\vu)= \mu\Big[\Grad \vu+(\Grad \vu)^{t}-\frac{2}{3}\Div \vu \tn{I}\Big]+\nu(\Div \vu)\tn{I},
\end{equation}
where $\mu=\mu(\vt)>0,\ \nu=\nu(\vt)\geq 0$, Lipschitz continuous functions in $\R^+$, are the shear and bulk viscosity coefficients, respectively, 
on which we assume
\begin{equation} \label{growth_mu}
\Un{\mu}(1+\vt)\leq\mu(\vt)\leq\Ov{\mu}(1+\vt),\quad 0\leq\nu(\vt)\leq\Ov{\nu}(1+\vt)
\end{equation}
for some positive constants $\underline{\mu},\overline{\mu},\overline{\nu}$,
and $\tn{I}$ is the identity matrix.
\subsection{Thermodynamic relations}

\medskip

\noindent {\bf Pressure and internal energy.}
We consider the pressure $\pi=\pi(\vr,\vt)$ with following form
\begin{equation}\label{defp}
\pi =\pi(\vr,\vt)=\pi_{c}(\vr)+\pi_{m}(\vr,\vt),
\end{equation}
where the molecular pressure $\pi_{m}$ obeys the Boyle law 
\begin{equation}\label{mol}
\pi_{m}=\sumkN\vr Y_k\vt=\vr\vt.
\end{equation}
It represents the pressure for an ideal mixture of $n$ species, 
with molar masses  equal to $1$. Moreover, without loss of generality, the gaseous constant equals one.
The first component of \eqref{defp}, $\pi_{c}$, is the so called {\it cold pressure}. 
We assume it in the form
\begin{equation} \label{pc}
\pi_{c}=\vr^{\gamma}, \quad \gamma>1.
\end{equation}
Indeed, a more general form of the cold pressure may be treated. The only important assumptions are that $\pi_c(\vr) \sim \vr^\gamma$ for $\vr$ large, $\pi\in C([0,\infty)) \cap C^1((0,\infty))$, strictly increasing in $\R^+$.
 
The specific total energy $E$ is a sum of the specific kinetic and specific internal energies
\begin{equation*}
E= E(\vr,\vu,\vt,Y_1,\ldots,Y_n)=\frac{1}{2}|\vu|^{2}+e(\vr,\vt,Y_1,\ldots,Y_n).
\end{equation*}
The internal energy consists of two components corresponding to the components of the pressure 
\begin{equation*}
e=e_{c}(\vr)+e_{m}(\vt,Y_1,\ldots,Y_n),
\end{equation*}
where the cold energy $e_c$ and the molecular internal energy $e_m$ are given by
\begin{equation*}
e_{c}=\frac{1}{\gamma-1}\vr^{\gamma-1},\qquad\qquad e_m= \sumkN Y_ke_k=\vt\sumkN c_{vk}Y_k.
\end{equation*}
Here, $c_{vk}$ are the mass constant-volume specific heats and can be different for different species. Under our assumption \eqref{eq_mass}
the constant-pressure specific heat, denoted by $c_{pk}$, equals
\begin{equation}\label{cpcv}
c_{pk}=c_{vk}+1,
\end{equation} 
and both $c_{vk}$ and $c_{pk}$ are assumed to be constant. 

\medskip

\noindent
{\bf Entropy.} According to the second law of thermodynamics, there exists a differentiable function called the {\it specific entropy of the mixture} $s(\vr,\vt,Y_1,\ldots,Y_n)$.  It can be expressed in terms of the partial specific entropies $s_{k}=s_k(\vr,\vt,Y_{k})$ of the $k$-th species
\begin{equation}\label{def_s}
s=\sumkN Y_k s_{k}.
\end{equation}
The Gibbs formula relates the differential of entropy  to the differential of energy, total density and mass fractions as follows
\begin{equation}\label{Gibbs}
\vt \D s=\D e+\pi\D\left({\frac {1}{\vr}}\right)-\sumkN g_{k}\D \vy,
\end{equation}
with the Gibbs functions
\begin{equation}\label{defg}
g_{k}=h_{k}-\vt s_{k}.
\end{equation}
Here $h_k=h_k(\vt)$, $s_{k}=s_{k}(\vr,\vt, Y_{k})$ denote the specific enthalpy and the specific entropy of the $k$-th species, respectively, with the following exact forms   
\eqh{h_k= c_{pk}\vt,\quad s_k= c_{vk}\log\vt-\log\vr-\log{Y_k},}
and we assume 
\begin{equation} \label{g_omega_pos}
-\sumkN g_k \omega_k \geq 0.
\end{equation}
The cold pressure and the cold energy correspond to  isentropic processes. Using (\ref{Gibbs}) it is possible to derive an equation for the specific entropy $s$
\begin{equation}\label{entropy}
\Div(\vr s\vu)+\Div\left( \frac{\vc{Q}}{\vt}-\sumkN \frac{g_{k}}{\vt}\vf_{k}\right)=\sigma,
\end{equation}
where $\sigma$ is the entropy production rate
\begin{equation} \label{sigma}
\sigma=\frac{{\tn S}:\Grad\vu}{\vt}-{\frac{\vc{Q}\cdot\Grad\vt}{\vt^{2}}}-\sumkN\vf_{k}\cdot\Grad\left({\frac{g_{k}} {\vt}}\right)-\frac{\sumkN g_{k}\vw_{k}}{\vt}.
\end{equation}

\subsection{The form of transport fluxes}

\medskip

\noindent {\bf Heat flux.}
The heat flux $\bf{Q}$ consists of two terms. 
The first one represents the transfer of energy due to the species molecular diffusion and the second one the Fourier law,
\begin{equation}\label{eq:heatd}
\vc{Q}=\sumkN h_k \vf_{k}+\vc{q},\quad \quad \vc{q}=-\kappa\Grad\vt,
\end{equation}
where $\kappa=\kappa(\vt)$ is the thermal conductivity coefficient on which we assume 
\begin{equation} \label{growth_kappa}
\underline{\kappa}(1+\vt^{m})\leq\kappa(\vt)\leq\Ov{\kappa}(1+\vt^{m})
\end{equation}
for some constants $m,\underline{\kappa},\overline{\kappa}>0$.

\medskip

\noindent {\bf Diffusion flux.}
The diffusion flux of the $k$-th species $\vf_{k}$ is given by
\begin{equation}
\vf_{k}=-\vr Y_k\sumlN D^\vr_{kl}\Grad Y_l,
\label{eq:diff1}
\end{equation}
where $D^\vr_{kl}=D^\vr_{kl}(\vr,\vt,Y_1,\ldots,Y_n)$, $k,l=1,\ldots,n$ are the multicomponent diffusion coefficients. 
The coefficients $\vr D^\vr_{kl}$ depend only on $\vt$ and $Y_1,\ldots,Y_n$ (see  \cite{VG}), therefore we introduce another matrix
$$\tn{D} = (D_{kl})_{k,l=1}^{n}=\vr (D^\vr_{kl})_{k,l=1}^{n}=(D_{kl}(\vt,Y_1,\ldots,Y_n))_{k,l=1}^{n}.$$
We denote by $N(\tn{D})$  the nullspace of the matrix $\tn{D}$, $R(\tn{D})$ its range, and ${\vec{Y}}^{\bot}$ is the orthogonal complement of $\vec{Y}$.
The diffusion matrix $\tn{D}$ has the following properties which are  discussed in \cite[Chapter 7]{VG}:
	\begin{equation}\label{prop}
	\begin{gathered}
		\tn{D}=\tn{D}^t,\quad
		N(\tn{D})=\R  \vec{Y},\quad
		R(\tn{D})={\vec{Y}}^{\bot},\\
		\tn{D} \quad\text{ is positive semidefinite over } \R^n.
	\end{gathered}
	\end{equation}
Note that we assumed $\vec{Y}=(Y_1,\ldots,Y_n)^t>0$.

Furthermore, the matrix $\tn{D}$ is homogeneous of a non-negative order with respect to $Y_1,\ldots,Y_n$ and $D_{ij}$ are differentiable functions of $\vt,Y_1,\ldots,Y_n$ for any $i,j\in\{1,\ldots,n\}$ such that 
$$
|D_{ij}(\vt,\vec{Y})| \leq C(\vec{Y}) (1+\vt^a)
$$
for some $a\geq 0$. Denoting ${\vec{U}}=(1,\ldots,1)^{t}$, the form of $\vf_k$ implies in particular $\{\vf_k\}_{k=1}^n \in {\vec{U}}^{\bot}$ which yields
\begin{equation} \label{sum_F}
\sumkN \vf_k=\vc{0}.
\end{equation} 
Therefore, since the species equations must sum to the continuity equation, we obtain
\begin{equation} \label{sum_omega}  
\qquad \sumkN\vw_k=0.
\end{equation}
%
\subsection{Entropy production rate}
Due to \eqref{prop} the matrix $\tn{D}$ is positive definite over $\vec{U}^\bot$.  
As we shall see now, this property is connected  with the positivity of entropy production rate $\sigma$
defined in \eqref{sigma}. 
Indeed, we have  
\begin{equation}
\nabla \Big(\frac{g_k}{\vt}+c_{pk}\log\vt\Big)=\nabla \log p_k,
\end{equation}
where $p_k=\vr Y_k\vt$.
Therefore \eqref{sigma} may be rewritten in the following form
\begin{equation} \label{sigma2}
\sigma=\frac{\tn{S}:\Grad\vu}{\vt}+{\frac{\kappa|\Grad\vt|^2}{\vt^{2}}}-\sumkN{\vf_{k}}\cdot\Grad\left(\log{p_{k}}\right)-{\frac{\sumkN g_{k}\vw_k}{\vt}}.
\end{equation}
Let us have a look on the structure of the third term. We have
\eq{\label{coerc}
-\sumkN{\vf_{k}}\cdot\Grad\left(\log{p_{k}}\right)=&-\sumkN\frac{\vf_{k}}{p_k}\cdot {\Grad p_{k}}
\\
=&-\sumkN\vf_{k}\cdot \lr{\frac{\Grad Y_{k}}{Y_k}+\frac{\Grad(\vr\vt)}{\vr\vt}}\quad [\mbox{due\ to\ }\eqref{sum_F}]\\
=&\sum_{k,l=1}^n D_{kl}\Grad Y_l\cdot\Grad Y_k\geq c\sumkN\frac{|\Grad Y_k|^2}{Y_k} \geq c \sumkN|\Grad Y_k|^2,
}
where we have used the fact that $\partial_{x_i} \vec{Y} \in \vec{U}^\bot$ for all $i \in \{1,2,\dots, n\}$ due to \eqref{brak} (cf. \cite[Lemma 7.6.1]{VG}) Note that the last inequality is due to the fact that $Y_k \leq 1$, therefore the second last term contains additional information about the mass fractions, but we do not exploit it.  
Now, \eqref{sigma2}, \eqref{coerc}, \eqref{vis_ten}
together with \eqref{g_omega_pos} yields $\sigma\geq 0$.


\section{Weak and variational entropy solutions. Main Results.}
We are now in a position to formulate the definition of weak solutions to our system.
\begin{df}\label{df1}
We say the set of functions $(\vr,\vu,\vt, \vec{Y})$ is a weak solution to problem 
(\ref{1.1}--\ref{conserva}) with assumptions stated above, 
provided 
\begin{itemize}
\item
$\vr \geq 0$ a.e. in $\Omega$, $\vr \in L^{6\gamma/5}(\Omega)$, $\int_{\Omega} \vr\dx=M$ 

\item
$\vu \in W^{1,2}_0(\Omega)$, $\vr |\vu|$ and $\vr |\vu|^2 \in L^{\frac{6}{5}}(\Omega)$

\item 
$\vt \in W^{1,2}(\Omega) \cap L^{3m}(\Omega)$, $\vr \vt, \vr\vt|\vu|, {\tn S}\vu, \kappa|\nabla \vt| \in L^1(\Omega)$

\item
$\vec{Y}\in W^{1,2}(\Omega)$, $Y_k \geq 0$ a.e. in $\Omega$, $\sumkN Y_k = 1$ a.e. in $\Omega$, $\vf_k\cdot \vc{n}|_{\partial \Omega}=0$
\end{itemize}
and the following integral equalities hold\\
$\bullet$ the weak formulation of the continuity equation
\begin{equation}\label{weak_cont}
\intO{\vr \vu\cdot\Grad\psi} = 0
\end{equation}
holds for any test function $\psi\in C^{\infty}(\Ov{\Omega})$;\\
$\bullet$ the weak formulation of the momentum equation
\eq{\label{weak_mom}
-\intO{\big(\vr\left(\vu\otimes\vu\right):\Grad\vcg{\vp}-\tn{S}:\Grad \vcg{\vp}\big)}-\intO {\pi \Div\vcg{\vp}}=\intO{\vr\vc{f}\cdot\vcg{\vp}}
}
holds for any test function $\vcg{\vp}\in  C^{\infty}_{0}(\Omega)$;\\
$\bullet$ the weak formulation of the species equations
\begin{equation}\label{weak_spe}
-\intO{ Y_{k}\vr\vu\cdot\Grad\psi}-\intO{\vf_k\cdot\Grad\psi}=\intO{\vw_{k}\psi}
\end{equation}
holds for any test function $\psi\in C^{\infty}(\Ov{\Omega})$ and for all $k=1,\ldots,n$;\\
$\bullet$ the weak formulation of the total energy balance
\eq{\label{weak_ene}
&-\intO{\left(\frac{1}{2}\vr|\vu|^{2}+\vr e\right)\vu\cdot\Grad\psi}+\intO{\kappa\Grad\vt\cdot\Grad\psi}
-\intO{\left(\sumkN{h_{k}\vf_{k}}\right)\cdot\Grad\psi}\\
&=\intO{\vr\vc{f}\cdot\vu\psi}-\intO{\lr{\tn{S}\vu}\cdot\Grad\psi}+\intO{\pi\vu\cdot\Grad\psi}-\intpO{L(\vt-\vt_{0})\psi}
}
holds for any test function $\psi\in  C^{\infty}(\Ov{\Omega})$.
\end{df}

The admissible range of $\gamma$ in the pressure law \eqref{pc} for which we are able to show 
existence of weak solutions in the above sense is limited mostly by the terms  
$\vr |\vu|^2\vu$ and $\tn{S} \vu$ in the weak formulation of total energy balance. Therefore,
following \cite{NoPo1}, \cite{NoPo2}
we introduce a slightly more general notion of variational entropy solutions to system \eqref{1.1}
which consist in replacing the weak formulation of the total energy balance by the weak formulation of the 
entropy inequality.\\ 
\begin{df}
We say the set of functions $(\vr,\vu,\vt, \vec{Y})$ is a variational entropy solution to problem 
(\ref{1.1}--\ref{conserva}) with assumptions stated above, 
provided
\begin{itemize}
\item
$\vr \geq 0$ a.e. in $\Omega$, $\vr \in L^{s\gamma}(\Omega)$ for some $s>1$, $\int_{\Omega} \vr\dx=M$ 

\item
$\vu \in W^{1,2}_0(\Omega)$, $\vr \vu \in L^{\frac{6}{5}}(\Omega)$

\item
$\vt \in W^{1,r}(\Omega) \cap L^{3m}(\Omega)$, $r>1$, $\vr \vt, \tn {S}:\frac{\nabla \vu}{\vt}, \kappa\frac{|\nabla \vt|^2}{\vt^2}, \kappa\frac{\nabla \vt}{\vt} \in L^1(\Omega)$
$\frac{1}{\vt} \in L^1(\partial \Omega)$

\item
$\vec{Y}\in W^{1,2}(\Omega)$, $Y_k \geq 0$ a.e. in $\Omega$, $\sumkN Y_k = 1$ a.e. in $\Omega$, $\vf_k\cdot \vc{n}|_{\partial \Omega}=0$

\end{itemize}
satisfy equations \eqref{weak_cont}--\eqref{weak_spe}, 
the following entropy inequality
\begin{multline} \label{entropy_ineq}
\int_{\Omega} \frac{ \tn {S} : \nabla \vu}{\vt}\psi \dx
+\int \kappa\frac{|\nabla \vt|^2}{\vt^2}\psi \dx
-\int_{\Omega}\sumkN \omega_k (c_{pk}-c_{vk} \log \vt + \log Y_k)\psi\dx\\
+\int_{\Omega} \psi \sum_{k,l=1}^n D_{kl}\nabla Y_k \cdot \nabla Y_l \dx
+\intpO{\frac{L}{\vt}\vt_0\psi} \leq
\int \frac{\kappa \nabla \vt \cdot \nabla \psi}{\vt} \dx
-\int_{\Omega} \vr s \vu \cdot \nabla \psi \dx\\
-\int_{\Omega} \log \vt \Big(\sumkN \vc{F}_k c_{vk}\Big) \cdot \nabla \psi \dx 
+\int_{\Omega} \Big(\sumkN \vc{F}_k \log Y_k\Big) \cdot\nabla \psi \dx 
+ \intpO{L\psi}
\end{multline}
for all non-negative $\psi \in C^\infty(\overline{\Omega})$
and the global total energy balance (i.e. \eqref{weak_ene} with $\psi \equiv 1$)
\begin{equation} \label{glob_ene}
\intpO{L(\vt-\vt_0)} = \intO{\vr \vc{f}\cdot \vu}.
\end{equation} 
\end{df}
Formally, the entropy inequality \eqref{entropy_ineq} is nothing but a weak formulation of \eqref{entropy}. We  will return to this in the part devoted to the formulation of the approximate solution, where we deduce the approximate entropy (in)equality from the approximate internal energy balance and the approximate momentum balance. Note, however, that we have here inequality instead of equality. This is a consequence of the fact that for sequences of functions which do not converge strongly but only weakly in some spaces we are not able to ensure the corresponding limit passages and we are obliged to use only the weak lower semicontinuity in some terms. Note further that  \eqref{entropy_ineq} does not contain all terms from \eqref{entropy}, some of them are missing. These terms are formally equal to zero due to assumptions that $\omega_k$ and $\vf_k$ sum up to zero. We removed them from the formulation of the entropy inequality due to the fact that we cannot exclude the situation that $\vr=0$ in some large portions of $\Omega$ (with positive Lebesgue measure), thus $\log \vr$ is not well defined there.  However, the variational entropy solution still has the property that any sufficiently smooth variational entropy solution in the sense above is a classical solution to our problem, provided the density is strictly positive in $\Omega$.  

We are now in position to formulate our main result.

\begin{thm} \label{t1}
Let $\gamma > 1$, $M>0$, $m > \max\{\frac 23, \frac{2}{3(\gamma-1)}\}$, $a < \frac{3m}{2}$. Let $\Omega \in C^2$. Then there exists at least one  variational entropy solution to our problem above. Moreover, $(\vr,\vu)$ is the renormalized solution to the continuity equation. 

In addition, if $m > \max\{1,\frac{2\gamma}{3(3\gamma-4)}\}$, $\gamma > \frac 43$, $a< \frac{3m-2}{2}$, then the solution is a weak solution in the sense above.
\end{thm}  

The second part of the main result, dealing with weak solutions, is an improvement of the result from \cite{GPZ}. This is connected with the fact that we will use finer estimates of the density before the last limit passage.

Note further that the assumptions on $\gamma$ and $m$ in both variational entropy and weak solution correspond to those which ensure the existence of the corresponding type of a solution for the steady compressible Navier--Stokes--Fourier system, cf. \cite{MPZ_HB}. Finally, recall that the pair $(\vr,\vu)$ is a renormalized solution to the continuity equation provided $\vu \in W^{1,2}(\Omega)$, $\vr \in L^{\frac 65}(\Omega)$ and for any $b \in C^{1}(0,\infty)\cap C([0,\infty))$, $b'(z) = 0$ for $z\geq M$ for some $M>0$
$$
\int_{\Omega} \Big(b(\vr) \Div \psi +  (b(\vr)-b'(\vr)\vr)\Div \vu \psi\Big) \dx =0
$$
for all $\psi \in C^\infty(\Ov{\Omega})$.

The weak solutions for the compressible Navier--Stokes equations were for the first time considered in the seminal monography by P.L. Lions \cite{Li_Book}. Their existence was shown for $\gamma >\frac 95$. Using more precise estimates of the density, the result was subsequently improved in the papers \cite{FrStWe}, \cite{JeNo} 
and \cite{Ji} to reach the existence of weak solutions for $\gamma>1$. The theory was applied to the compressible Navier--Stokes--Fourier system in the series of papers \cite{NoPo1}, \cite{NoPo2} (here, the notion of variational entropy solutions in the steady case was introduced) and \cite{JeNoPo}. See also \cite{MPZ_HB} for further details.

The system of equations describing the flow of chemically reacting, heat conducting  gaseous mixture was considered firstly in the evolutionary case in the context of variational entropy solutions in \cite{FPT}, however, with Fick's law. A more general multicomponent diffusion flux was in the context of weak solutions considered in \cite{MuPoZa} and in the steady regime in \cite{GPZ}.

\section{Approximation}
Following \cite{GPZ} we will prove our main results introducing five steps of approximation.
The first four are connected with small parameters $\delta > \varepsilon > \lambda > \eta >0$ and the last one, connected with a positive integer $N$, 
is the Galerkin approximation for the velocity. 

Precisely, we introduce the approximation of diffusion flux $\vf_k$:
\begin{equation} \label{def_jk}
{\vc J}_k =-\sumlN Y_k Y_l\widehat{ D}_{kl}(\vt,\vec{Y})\Grad Y_l/Y_l
-\big(\ep(\vr+1) Y_k+\lambda\big)\Grad Y_k/Y_k,\\
\end{equation}
with
\begin{equation} \label{dkl_aprox} 
\widehat {D}_{kl}(\vt,\vec{Y}) = \frac{1}{(\sigma_Y+\ep)^r}  D_{kl}(\vt,\vec{Y}),
\end{equation}
where $\sigma_Y=\sumkN Y_k$. The reason for this notation is that, unless we let $\lambda \to 0^+$, it is not clear whether $\sigma_Y=1$. We only know that $Y_k \geq 0$.
Furthermore, we introduce a regularization of the stress tensor 
\begin{equation}
\tn{S}_{\eta}={\frac{\mu_{\eta}(\vt)}{1+\eta\vt}}\left[\Grad \vu+(\Grad\vu)^t-\frac{2}{3}\Div \vu \, \tn{I}\right]+{\frac{\nu_\eta(\vt)}{1+\eta\vt}}\lr{\Div \vu}\tn{I},
\end{equation}
where $\mu_{\eta},\nu_{\eta}$ are standard mollifications of the viscosity functions.
Next,
\begin{equation} \label{def_kappad}
\kappa_{\delta,\eta} = \kappa^\eta + \delta \vt^B + \delta \vt^{-1}
\end{equation} 
is a regularization of heat conductivity coefficient with $B>0$ sufficiently large which will
be determined later and $\kappa^\eta$ is the mollification of the heat conductivity.    
Compared to \cite{GPZ} we introduce a  
minor modification in the approximation, namely we approximate the fractional entropies with 
\begin{equation} \label{skl}
s_k^\lambda=c_{vk}\log \vt - \log Y_k - \log(\vr+\sqrt{\lambda}). 
\end{equation}
Analogously, we denote 
$$
g_k^\lambda = c_{pk} \vt - \vt s_k^\lambda, \quad s^\lambda=\sumkN Y_k s_k^\lambda.
$$
This modification will enable us to pass to the limit with $\lambda$ in the weak formulation of the entropy inequality,
on the other hand it is harmless for crucial \emph{a priori} estimates for the full approximation.

We are now ready to formulate the approximate problem involving five above mentioned parameters.
Let $\{\vc{w}_n\}_{n=1}^{\infty}$ be an orthogonal basis of $W^{1,2}_0(\Omega)$ such that 
$\vc{w}_i \in W^{2,q}(\Omega)$ for $q<\infty$ (we can take for example eigenfunctions of the Laplace
operator with Dirichlet boundary conditions).  
At the level of full approximation we want to show existence of a set of functions
$(\vr_{N,\eta,\lambda,\ep,\delta},\vu_{N,\eta,\lambda,\ep,\delta},\vec{Y}_{N,\eta,\lambda,\ep,\delta},\vt_{N,\eta,\lambda,\ep,\delta})$ 
(from now on we skip the indices) such that\\
$\bullet$ the approximate continuity equation
\eq{\label{cont_aprox}
\ep\vr+\Div (\vr \vu) &= \ep\Delta\vr+\ep \Ov{\vr},\\
\Grad\vr\cdot\vc{n}\bo&=0,
}
where $\bar \vr = \frac{M}{|\Omega|}$,
is satisfied pointwisely\\
$\bullet$ the Galerkin approximation for the momentum equation (note that the convective term reduces to the standard form provided $\Div(\vr\vu)=0$, even in the weak sense)
\begin{multline}\label{mom_aprox}
\intOB{\frac{1}{2}\vr\vu\cdot\Grad\vu\cdot\vc{w}-\frac{1}{2}\vr\left(\vu\otimes\vu\right):\Grad\vc{w}+\tn{S}_{\eta}:\Grad\vc{w}}\\-\intO{(\pi+\delta\vr^{\beta}+\delta\vr^{2})\Div\vc{w}}=\intO{\vr \vc{f}\cdot\vc{w}}
\end{multline}
is satisfied for each test function $\vc{w}\in X_{N}$, where $\vu \in X_N$,
$X_N={\rm span}\{\vc{w}_i\}_{i=1}^N$, and $\beta>0$ is large enough\\ 
$\bullet$ the approximate species mass balance equations 
\eq{\label{spec_aprox}
\begin{array}{c}
\Div \vc{J}_k=\vw_{k}+\ep \Ov{\vr}_k-\ep Y_k\vr-\Div(Y_k\vr\vu )+\ep\Div(Y_k\Grad\vr) -\sqrt{\lambda} \log Y_k , \\
\vc{J}_k \cdot \vc{n}\bo = 0 
\end{array}
}
are satisfied pointwisely,
where $\sumkN \bar \vr_k=\bar \vr$, for example we take $\bar \vr_k=\frac{\bar \vr}{n}$
\\ 
$\bullet$ the approximate internal energy balance
\eq{\label{en_aprox}
-\Div\left(\kappa_{\delta,\eta}{\frac{\ep+\vt}{\vt}}\Grad \vt\right)
= &-\Div(\vr e\vu)-\pi\Div\vu +{\frac{\delta}{\vt}} 
+\tn{S}_{\eta}:\Grad{\vu} \\
&+\delta\ep(\beta\vr^{\beta-2}+2)|\Grad\vr|^{2}-\Div\left(\vt\sumkN c_{vk} \vc{J}_k\right)
}
with the boundary condition
\begin{equation}\label{apbound}
\kappa_{\delta,\eta} \frac{\ep+\vt}{\vt}\Grad\vt\cdot\vc{n}\bo+(L+\delta\vt^{B-1})(\vt-\vt_{0}^\eta)+\ep \log\vt +\lambda \vt^{\frac B2} \log \vt=0
\end{equation}
is satisfied pointwisely, where $\vt_0^\eta$ is a smooth, strictly positive approximation of $\vt_0$
and $\kappa_{\delta,\eta}$ is as above.

It remains to formulate the approximate entropy inequality for the purpose of showing existence
of variational entropy solutions. Note that the entropy inequality (or rather equality on this level of approximation) is not an additional assumption, but a consequence of the approximate relations above.
 
\begin{rmk}\label{sqrt}
Note that there is one more change with respect to paper \cite{GPZ}, namely we have in \eqref{spec_aprox} in the last term on the right-hand side $\sqrt{\lambda}$ instead of $\lambda$. This is connected with the limit passage $\lambda\to 0^+$ in the weak formulation of the entropy inequality. It is an easy matter to check that the proof in \cite{GPZ} would work also for this approximation.
\end{rmk}

\subsection{Approximate entropy inequality}

We now deduce the form of the approximate entropy inequality. Even though the computations below are rather formal (and require certain regularity of all functions), it can be verified that the regularity enjoyed by the approximate solutions is enough for the entropy equality to hold.
  
Recalling the form of internal energy and pressure we observe that 
$$
\Div (\vr e \vu)+\pi \Div \vu = \vr \vu \cdot \nabla \Big(\frac{\vr^{\gamma-1}}{\gamma-1}
+\vt \sumkN c_{vk}Y_k \Big) + e \Div(\vr\vu) + (\vr^{\gamma-1}+\vt)\vr \Div \vu 
$$$$
=\vr \vu \cdot \nabla \Big(\vt\sumkN c_{vk}Y_k\Big)-\vt \vu \cdot \nabla \vr+(\vr^{\gamma-1}+\vt+e)\Div (\vr \vu).
$$
Therefore, multiplying the approximate internal energy balance \eqref{en_aprox} by $\frac{\psi}{\vt}$ 
and integrating over $\Omega$ we get 
\begin{multline} \label{s1}
\int_{\Omega} \frac{\kappa_{\delta,\eta}(\varepsilon+\vt)\nabla \vt \cdot \nabla \psi}{\vt^2} \dx
-\int_{\Omega} \kappa_{\delta,\eta}\frac{(\varepsilon+\vt)}{\vt}\frac{|\nabla \vt|^2}{\vt^2}\psi \dx
\\
+ \int_{\partial \Omega} \frac{\psi}{\vt}\Big[ (L+\delta\vt^{B-1}) (\vt-\vt_0^\eta)
+\varepsilon \log \vt + \lambda \vt^{B/2} \log \vt\Big] \dS\\
\underbrace{-\int \sumkN h_k {\vc J}_k \cdot \nabla \Big(\frac{\psi}{\vt}\Big) \dx}_{B1}
\underbrace{+\int \frac{\psi}{\vt}\Big[\vr \vu \cdot \big(\nabla \vt \sumkN c_{vk}Y_k + \vt \sumkN c_{vk}\nabla Y_k\big) - \vt \vu \cdot \nabla \vr\Big]\dx}_{D} \\
+\int_{\Omega}\ep(\Delta \vr+\bar \vr-\vr)(\vr^{\gamma-1}+e+\theta)\frac{\psi}{\vt}\dx
-\int_{\Omega}\frac{\delta \psi}{\vt^2}\dx
-\int_{\Omega}\frac{\psi \tn{S}_{\eta}:\nabla \vu}{\vt}\dx\\
-\int_{\Omega}\vt\Big[\sumkN(\ep(\vr+1) Y_k+\lambda)\frac{\nabla Y_k}{Y_k}\Big]\cdot \nabla\Big(\frac{\psi}{\vt}\Big)\dx
-\int_{\Omega} \frac{\delta \varepsilon (\beta \vr^{\beta-2}+2)|\nabla \vr|^2\psi}{\vt}\dx=0.
\end{multline}
Taking the sum over $k$ of the approximate species equations \eqref{spec_aprox} 
multiplied by $-\frac{g_k^\lambda \psi}{\vt}$ we get 
\begin{multline} \label{s2}
\underbrace{\int_{\Omega} \psi \sumkN \Big(Y_k \vr \vu \cdot \nabla \Big(\frac{g_k^\lambda}{\vt}\Big)\Big)\dx 
+ \int_{\Omega} \sumkN \Big(\frac{g_k^\lambda}{\vt} Y_k\Big) \vr \vu \cdot  \nabla \psi\dx}_{C} 
\underbrace{+ \int_{\Omega} \sumkN {\vc J}_k \cdot \nabla \Big(\frac{g_k^\lambda \psi}{\vt}\Big)\dx}_{B2}\\
-\varepsilon \int_{\Omega} \sumkN Y_k \nabla \vr \cdot \nabla \Big(\frac{g_k \psi}{\vt}\Big) \dx
-\sqrt{\lambda} \int_{\Omega} \frac{\psi}{\vt}\sumkN \log Y_k g_k^\lambda \dx +
\varepsilon \int_{\Omega} \sumkN (\bar \vr_k - Y_k \vr) \frac{g_k^\lambda \psi}{\vt}\dx\\
= - \int_{\Omega} \sumkN \frac{g_k^\lambda \omega_k \psi}{\vt}\dx.
\end{multline}
The definition of $g_k^\lambda$ yields 
$$
\begin{aligned}
B1+B2 =& -\int_{\Omega} \Big(\sumkN c_{pk}{\vc J}_k\Big)\cdot \nabla \psi \dx
+ \int_{\Omega}\psi\Big(\sumkN c_{pk}{\vc J}_k\Big)\cdot \nabla \log \vt\dx \\
&+\int_{\Omega}\frac{\nabla \psi}{\vt} \cdot \Big(\sumkN{\vc J}_k g_k^\lambda\Big)\dx 
+\int_{\Omega}\psi\sumkN {\vc J}_k\cdot\nabla \Big(\frac{g_k^\lambda}{\vt}\Big)\dx\\
=& - \int_{\Omega} \Big(\sumkN {\vc J}_k s_k^\lambda\Big) \nabla \psi \dx + \int_{\Omega} \Big(\sumkN c_{pk} {\vc J}_k\Big) \cdot \nabla (\log \vt)  \psi \dx \\
&+ \int_{\Omega} \psi \sumkN {\vc J}_k \cdot \nabla \Big(\frac{g_k^\lambda}{\vt}\Big) \dx,
\end{aligned}
$$
and 
$$
C = \int_{\Omega} \psi \vr \vu \cdot \Big(\sumkN Y_k  \nabla \Big(\frac{g_k^\lambda}{\vt}\Big)\Big) \dx
- \int_{\Omega} \psi \Div \Big( \vr\vu \sumkN Y_k (c_{pk}-s_k^\lambda) \Big)\dx.
$$
Rewriting the second term by virtue of the approximate continuity equation
$$
\begin{aligned}
C &= \int_{\Omega} \psi \vr \vu \cdot \Big(\sumkN Y_k \nabla \Big(\frac{g_k^\lambda}{\vt}\Big)\Big)\dx 
- \int_{\Omega} \psi \vr \vu \cdot \sumkN c_{pk} \nabla Y_k \dx
\\
+& \int_{\Omega} \psi \vr \vu \cdot \nabla\Big(\sumkN Y_k s_k^\lambda\Big)\dx
- \int_{\Omega} \psi\Div(\vr \vu) \sumkN  Y_k (c_{pk}-s_k^\lambda) \dx.  
\end{aligned}
$$
Finally we have 
$$
D = \int_{\Omega} \psi \vr \vu \cdot \Big(\sumkN c_{vk} \nabla Y_k\Big) \dx
+ \int_{\Omega} \psi \vr \vu\cdot  \Big(\sumkN \nabla(c_{vk}\log \vt) Y_k\Big)  \dx 
- \int_{\Omega} \psi \vr \vu \cdot \nabla (\log \vr) \dx. 
$$
Substituting $c_{vk} \log \vt = s_k^{\lambda}+\log (\vr+\sqrt{\lambda})+\log Y_k$ to the second term yields 
\begin{multline} \label{cd}
C+D = -\int_{\Omega}\psi\vr\vu\cdot \Big(\sumkN  \nabla s_k^\lambda Y_k\Big)\dx 
-\int_{\Omega}\psi\vr\vu\cdot\Big(\sumkN \nabla Y_k c_{pk}\Big)\dx
+\int_{\Omega}\psi\vr\vu \cdot \nabla s^\lambda \dx\\
+\int_{\Omega}\psi\vr\vu\cdot\Big(\sumkN \nabla Y_k c_{vk}\Big)\dx
+\int_{\Omega}\psi\vr\vu\cdot \Big(\sumkN \nabla s_k^\lambda Y_k\Big)\dx \\
+\int_{\Omega}\psi\vr\vu\cdot \big(\Big(\sumkN Y_k\Big)\nabla \log (\vr+\sqrt{\lambda})-\nabla \log \vr\big)\dx\\
+\int_{\Omega}\psi\vr\vu\cdot \Big(\sumkN \nabla Y_k\Big)\dx
- \int_{\Omega} \sumkN \psi Y_k (c_{pk}-s_k^\lambda) \Div(\vr \vu)\dx\\
=\int_{\Omega} \psi \vr \vu \cdot \nabla s^\lambda \dx 
+ \int_{\Omega}  \psi \vr \vu \cdot \Big(\Big(\sumkN Y_k\Big)\nabla \log (\vr+\sqrt{\lambda}) - \nabla \log \vr\Big)\dx\\
- \int_{\Omega} \sumkN \psi Y_k c_{pk} \Div(\vr \vu)\dx
+  \int_{\Omega} \psi \sumkN (Y_k s_k^\lambda) \Div(\vr \vu)\dx.
\end{multline}
Integrating in the first term by parts we get 
$$
\int_{\Omega} \psi \vr \vu \cdot \nabla s^\lambda \dx  = 
-\int_{\Omega} \psi s^\lambda \Div(\vr \vu)\dx - \int_{\Omega}s^\lambda \vr \vu \cdot \nabla \psi \dx. 
$$
The first term cancels with the last term from \eqref{cd} and applying the approximate 
continuity equation yields
\begin{equation} \label{cd1}
\begin{aligned}
C+D =& -\int_{\Omega} s^\lambda \vr \vu \cdot \nabla \psi \dx
+ \int_{\Omega}  \psi \vr \vu \cdot \Big[ \Big(\sumkN Y_k\Big)\nabla \log (\vr+\sqrt{\lambda}) \\
&- \nabla \log \vr\Big]\dx
- \varepsilon \int_{\Omega} \psi \sumkN Y_k c_{pk} (\Delta \vr + \bar \vr - \vr) \dx. 
\end{aligned}
\end{equation}  
With the above considerations we are ready to formulate the approximate entropy inequality 
which at this stage can be still written as equality. Namely, adding \eqref{s1} and \eqref{s2}
we arrive at 
\begin{multline} \label{entropy_ap_0} 
\int \frac{\kappa_{\delta,\eta}(\varepsilon+\vt)\nabla \vt \cdot \nabla \psi}{\vt^2} \dx
-\int \kappa_{\delta,\eta}\frac{(\varepsilon+\vt)}{\vt}\frac{|\nabla \vt|^2}{\vt^2}\psi \dx
\\
+ \int_{\partial \Omega} \frac{\psi}{\vt}\Big[ (L+\delta\vt^{B-1}) (\vt-\vt_0^\eta)
+\varepsilon \log \vt + \lambda \vt^{B/2}\log \vt\Big]\, \dS \\
-\int_{\Omega}\frac{\delta \psi}{\vt^2}\dx
-\int_{\Omega} \frac{\psi \tn{S}_\eta : \nabla \vu}{\vt}\dx 
+\int_{\Omega}\ep(\Delta \vr+\bar \vr-\vr)(\vr^{\gamma-1}+e+\theta)\frac{\psi}{\vt}\dx \\
+ \int_{\Omega}\sumkN \frac{g_k^\lambda \omega_k \psi}{\vt}\dx 
-\int_{\Omega} \vr s^\lambda \vu \cdot \nabla \psi \dx\\
\underbrace{
- \int_{\Omega} \Big(\sumkN {\vc J}_k s_k^\lambda\Big)\cdot \nabla \psi \dx 
+ \int_{\Omega} \psi \nabla (\log \vt) \cdot \Big(\sumkN c_{pk} {\vc J}_k\Big) \dx
+ \int_{\Omega} \psi \sumkN {\vc J}_k \cdot \nabla \Big(\frac{g_k^\lambda}{\vt}\Big) \dx}_{J}\\
+ \int_{\Omega}  \psi \vr \vu\cdot  \Big(\Big(\sumkN Y_k\Big)\nabla \log (\vr+\sqrt{\lambda}) - \nabla \log \vr\Big)\dx
- \varepsilon \int_{\Omega} \psi \sumkN Y_k c_{pk} (\Delta \vr + \bar \vr - \vr) \dx\\
-\varepsilon \int_{\Omega} \sumkN Y_k \nabla \vr \cdot \nabla \Big(\frac{g_k \psi}{\vt}\Big) \dx
-\sqrt{\lambda} \int_{\Omega} \frac{ \psi}{\vt} \sumkN \log Y_k g_k \dx 
+\varepsilon \int_{\Omega} \sumkN (\bar \vr_k - Y_k \vr) \frac{g_k \psi}{\vt}\dx\\ 
-\int_{\Omega}\vt\Big(\sumkN(\ep(\vr+1) Y_k+\lambda)\frac{\nabla Y_k}{Y_k}\Big)\cdot \nabla\Big(\frac{\psi}{\vt}\Big)\dx 
-\int_{\Omega}\frac{\delta \varepsilon (\beta \vr^{\beta-2}+2)|\nabla \vr|^2\psi}{\vt}\dx=0.
\end{multline}
Taking into account \eqref{def_jk}, the sum of terms containing ${\vc J}_k$ equals (we use the notation $\widehat{\vc{F}}_k$ $= -\sumkN Y_k \widehat{D}_{kl} (\vt,\vec{Y}) \nabla Y_l$) 
\begin{multline}
J = - \int_{\Omega} \Big(\sumkN \widehat{\vc F}_k s_k^\lambda\Big)\cdot \nabla \psi \dx 
+ \int_{\Omega} \psi \nabla (\log \vt) \cdot \Big(\sumkN c_{pk} \widehat{\vc F}_k\Big) \dx
+ \int_{\Omega} \psi \sumkN \widehat{\vc F}_k \cdot \nabla \Big(\frac{g_k^\lambda}{\vt}\Big) \dx\\
\underbrace{-\int_{\Omega}\psi\sumkN c_{pk}\big(\ep(\vr+1)Y_k+\lambda\big)\frac{\nabla Y_k}{Y_k}\cdot \nabla\log\vt\dx
+\int_{\Omega}\psi\sumkN(\ep(\vr+1)Y_k+\lambda)\frac{\nabla Y_k}{Y_k}c_{vk}\nabla\log\vt\dx}_{J_1}\\
+\int_{\Omega}\sumkN\big(\ep(\vr+1)Y_k+\lambda\big)s_k^\lambda \frac{\nabla Y_k}{Y_k}\cdot\nabla\psi\dx
-\int_{\Omega}\psi\sumkN\big(\ep(\vr+1)Y_k+\lambda\big)\frac{\nabla Y_k}{Y_k}\cdot\nabla\log(\vr+\sqrt{\lambda})\dx\\
-\int_{\Omega}\psi\sumkN(\ep(\vr+1)Y_k+\lambda)\Big|\frac{\nabla Y_k}{Y_k}\Big|^2\dx.
\end{multline}
Recalling \eqref{cpcv} we have
\begin{equation}
J_1=-\int_{\Omega}\psi\sumkN(\ep(\vr+1)Y_k+\lambda)\frac{\nabla Y_k}{Y_k}\cdot\nabla\log\vt\dx.
\end{equation}
The second last term in \eqref{entropy_ap_0} reads
$$
-\int_{\Omega}\sumkN (\ep(\vr+1)Y_k+\lambda)\frac{\nabla Y_k}{Y_k}\cdot\nabla\psi\dx
+\int_{\Omega}\psi\sumkN (\ep(\vr+1)Y_k+\lambda)\frac{\nabla Y_k}{Y_k}\cdot\nabla\log\vt\dx.
$$
Now, the second term above cancels with $J_1$.

For the purpose of the passage to the limit it is better to rewrite the above formulation in 
the following way, using the fact that $\sumkN \widehat{\vc{F}}_k = \mathbf{0}$ and $\sumkN \omega_k =0$
\begin{multline} \label{entropy_aprox}
\int_{\Omega} \frac{\psi \tn {S}_\eta : \nabla \vu}{\vt}\dx
+\int_\Omega \kappa_{\delta,\eta}\frac{(\varepsilon+\vt)}{\vt}\frac{|\nabla \vt|^2}{\vt^2}\psi \dx
-\int_{\Omega}\omega_k (c_{pk}-c_{vk} \log \vt + \log Y_k)\psi\dx
\\
+\int_{\Omega}\frac{\delta \psi}{\vt^2}\dx-\int_{\Omega} \psi \sumkN \widehat{\vc F}_k \cdot \nabla \log Y_k \dx
+ \int_{\partial \Omega} \frac{\psi}{\vt}(L+\delta \vt^{B-1})\vt_0^\eta \, \dS\\ 
+\int_{\Omega}\frac{\delta \varepsilon (\beta \vr^{\beta-2}+2)|\nabla \vr|^2\psi}{\vt}\dx
+\int_{\Omega}\psi\sumkN(\ep(\vr+1)Y_k+\lambda)\Big|\frac{\nabla Y_k}{Y_k}\Big|^2\dx \\
=\int_{\Omega} \frac{\kappa_{\delta,\eta}(\varepsilon+\vt)\nabla \vt \cdot \nabla \psi}{\vt^2} \dx
-\int_{\Omega} \vr s^\lambda \vu \cdot \nabla \psi \dx
- \int_{\Omega} \sumkN (c_{vk}\log \vt -\log Y_k)  \widehat{\vc F}_k \cdot \nabla \psi \dx\\ 
+ \int_{\Omega} \psi \vr \vu \cdot \Big(\big(\sumkN Y_k\big) \nabla \log (\vr+\sqrt{\lambda}) - \nabla \log \vr\Big)\dx 
- \varepsilon \int_{\Omega} \psi \sumkN Y_k c_{pk} (\Delta \vr + \bar\vr - \vr) \dx \\
+ \int_{\partial \Omega} \frac{\psi}{\vt}\big( (L+\delta\vt^{B-1})\vt + \varepsilon \log \vt + \lambda \vt^{B/2}\log \vt\big)\, \dS
-\varepsilon \int_{\Omega} \sumkN Y_k \nabla \vr \cdot \nabla \Big(\frac{g_k^\lambda \psi}{\vt}\Big) \dx\\
-\sqrt{\lambda} \int_{\Omega} \Big(\sumkN g_k^\lambda  \log Y_k\Big) \frac{\psi}{\vt}\dx 
+\int_{\Omega}\ep(\Delta \vr+\bar \vr-\vr)(\vr^{\gamma-1}+e+\theta)\frac{\psi}{\vt}\dx\\
+\varepsilon \int_{\Omega} \sumkN (\bar \vr_k - Y_k \vr) \frac{g_k^\lambda \psi}{\vt}\dx 
-\int_{\Omega}\sumkN (\ep(\vr+1)Y_k+\lambda)\frac{\nabla Y_k}{Y_k}\cdot\nabla\psi\dx\\
+\int_{\Omega}\sumkN\big(\ep(\vr+1)Y_k+\lambda\big)s_k^\lambda\frac{\nabla Y_k}{Y_k}\cdot \nabla\psi\dx
-\int_{\Omega}\psi\sumkN(\ep(\vr+1)Y_k+\lambda)\frac{\nabla Y_k}{Y_k}\cdot \nabla\log(\vr+\sqrt{\lambda})\dx.
\end{multline}
Letting formally $\eta\to 0^+$, $\lambda\to 0^+$,  $\varepsilon\to 0^+$ and $\delta\to 0^+$,  we obtain \eqref{entropy_ineq} with equality. 
However, in rigorous limit passages 
we will have to apply the weak lower semicontinuity of norms leading to inequality instead of the equality.

\subsection{Existence of solutions for the Galerkin approximation.}
The existence of a solution can be proved exactly as in \cite[Theorem 5.2]{GPZ}.  The proof is based on the following ideas:
\begin{itemize}
\item the existence is proved by means of a version of the Schauder fixed point theorem for a suitably defined operator
\item instead of the temperature $\vt$ and the mass fractions $Y_k$ we look for their logarithms to ensure their positiveness
\item the a priori estimates are deduced from the entropy inequality \eqref{entropy_aprox} with $\psi \equiv 1$, the ``total" energy balance integrated over $\Omega$ (i.e. \eqref{mom_aprox} with $\vc{w}=\vu$ and the internal energy balance \eqref{en_aprox} integrated over $\Omega$), the approximate continuity equation \eqref{cont_aprox} and the Galerkin approximation of the momentum balance \eqref{mom_aprox} with $\vc{w}=\vu$     
\end{itemize}
We can verify the following result
\begin{thm}\label{T3}
Let $\delta$, $\ep$, $\lambda$ and $\eta$ be positive numbers and $N$ a positive integer. 
Let $\Omega \in C^2$. 
Then there exists a solution to system (\ref{cont_aprox}--\ref{en_aprox}) such that 
$\vr\in W^{2,q}(\Omega)$ $\forall q<\infty$, $\vr\geq0$ in $\Omega$, $\intO{\vr}=M$, $\vu\in X_N$, $\vec{Y}\in W^{1,2}(\Omega)$ with $\log Y_k \in W^{2,q}(\Omega)$ $\forall q<\infty$, $Y_k>0$ a.e. in $\Omega$ and $\vt\in W^{2,q}(\Omega)$ $\forall q<\infty$, $\vt\geq C(N)>0$.
Moreover, this solution satisfies the entropy equation \eqref{entropy_aprox} and the following 
estimate
\begin{multline} \label{est_1}
\sqrt{\lambda}\sumkN \Big(\|Y_k\|_{1,2}+\Big\|\frac{\nabla Y_k}{Y_k}\Big\|_2+\lambda^{-1/4}\|\log Y_k\|_2\Big) + \sumkN \Big\|\frac{|\nabla Y_k|^2}{Y_k} \Big\|_1 
+\|\nabla \vt^{B/2}\|_2 + \Big\|\frac{\nabla \vt}{\vt^{2}}\Big\|_2 \\
+ \Big\|\frac{\nabla \vr}{\sqrt{\vr+\sqrt{\lambda}}}\Big\|_2
+\|\vt^{-2}\|_1+\|\vt\|_{B,\partial \Omega}+\Big\|\frac{\log \vt}{\vt}\Big\|_{1,\partial \Omega} 
+\|\nabla^2 \vr\|_2 + \|\vu\|_{1,2} + \|\nabla \vr\|_6 \leq C,  
\end{multline}
where $C$ is independent of $N$. 
\end{thm}
Note that the bound on $\log Y_k$ in $L^2$ appears in \eqref{est_1} due to the presence of the term $-\sqrt{\lambda} (\int_\Omega \sumkN g_k^\lambda \log Y_k) \frac{1}{\vt} \dx$ on the right-hand side of \eqref{entropy_aprox}. The term $\frac{|\nabla \vr|^2}{\vr+\sqrt{\lambda}}$ appears due to the 7th term on the right-hand side of \eqref{entropy_aprox}.
 
\begin{rmk}
In the entropy inequality particular attention should be paid to terms containing logarithms,
since at the level of approximation we should avoid infinities in the entropy formulation.
We overcome this difficulty constructing the approximate temperature and $Y_k$ as 
exponential functions and possible singularities in $\log \vr$ are avoided due to definition 
of $s_k^\lambda$ \eqref{skl}. Thus we know that all the quantities in the approximate entropy 
equation \eqref{entropy_aprox} are finite. However, we must control that these terms remain finite throughout all passages below.
\end{rmk} 

\section{Limit passages I}

In this section we will study the limit passages $N \to \infty$, $\eta\to 0^+$, $\lambda \to 0^+$ and $\varepsilon \to 0^+$. Most of the arguments will be similar to \cite{GPZ} and the references therein, therefore we will mostly skip them and we will concentrate mostly on the new aspect, i.e. the entropy (in)equality which must hold (possibly modified) after each limit passage.
 
\subsection{Limit passages $N \to \infty$ and $\eta \to 0$}
We start with $N \to \infty$. 
At this stage the estimates copy exactly \cite{GPZ}, hence we may follow the arguments 
there. Note that, except the quadratic term in $\nabla \vu_N$ on the right-hand side (rhs) 
of the internal energy balance \eqref{en_aprox}, the limit passages are easy to perform. 
To get also the convergence of this term we use the fact that due to the $\eta$-approximation 
of the stress tensor we may use as test function $\vu$ in the limit version of the 
momentum equation \eqref{mom_aprox} 
and get
$$
\lim_{N\to \infty} \int_\Omega \tn{S}_\eta (\vt_N,\nabla \vu_N):\nabla\vu_N \dx = \int_\Omega \tn{S}_\eta (\vt,\nabla \vu):\nabla\vu \dx 
$$  
due to the energy equality. This equality even implies that $\nabla \vu_N\to \nabla \vu$ strongly in $L^2(\Omega)$, however, we do not use this information here.

Next we deal with the entropy inequality. 
In the first two terms in \eqref{entropy_aprox} we use the weak lower semicontinuity of $L^2$ norm with respect to weak convergence in $L^2$ 
(see \cite{NoPo1} for details). We have to restrict ourselves to non-negative test functions $\psi$ and get  
\begin{equation} \label{weak_low1}
\lim_{N \to \infty} \int_{\Omega} \frac{\tn{S}_\eta^N:\nabla \vu_N}{\vt_N}\psi\dx \geq
\int_{\Omega} \frac{\tn{S}_\eta:\nabla \vu}{\vt}\psi\dx 
\end{equation}
and 
\begin{equation} \label{weak_low2}
\lim_{N \to \infty} \int_{\Omega} \kappa_{\delta,\eta} \frac{(\varepsilon+\vt_N)|\nabla \vt_N|^2}{\vt_N^3}\psi\dx \geq
\int_{\Omega} \kappa_{\delta,\eta} \frac{(\varepsilon+\vt)|\nabla \vt|^2}{\vt^3}\psi\dx.
\end{equation}
In the other terms we can pass to the limit due to estimates \eqref{est_1}, however, we comment some of the limits in more details. 
Notice that in the 8th term on the rhs  the part with $\log \vr$ does not cause any troubles due to the control of $\log Y_k$ in $L^6$. However, in the subsequent limit passages, we will have to use another argument here. Similarly we may treat all other terms containing $\log Y_k$. The terms containing $\log \vr$ are either multiplied by $\vr$, or they are in fact in the form $\log (\vr+\sqrt{\lambda})$ and cause no troubles at this moment. Therefore the entropy inequality (we loose equality here) of the form \eqref{entropy_aprox} holds true. Note only that the test functions $\psi$ must be non-negative and we have inequality ($\leq$) instead of the equality sign in \eqref{entropy_aprox}.

The next step is the passage $\eta \to 0^+$. Since we have no information to ensure the strong convergence of the quadratic term on the rhs of the internal energy balance \eqref{en_aprox}, we have to replace it by the total energy inequality. To this aim, we sum \eqref{en_aprox} with the kinetic energy balance, i.e. \eqref{mom_aprox} with the test function $\vc{w}=\vu \psi$ (this was not possible on the level of Galerkin approximation), and we obtain
\begin{multline} \label{tot_en_aprox}
-\intO{\Big[\vr e+ \frac 12 \vr |\vu|^2 + (\pi+ \delta \vr^\beta + \delta \vr^2)\Big]\vu\cdot \Grad \psi}  
\\
-\intO{\Big(\tn{S}_\eta \vu \cdot \Grad \psi + \delta \vt^{-1} \psi\Big)} 
+\intO{\kappa_{\delta,\eta}{\frac{\ep+\vt}{\vt}}\Grad\vt\cdot\Grad \psi} \\
+ \int_{\partial \Omega}\big[(L+\delta \vt^{B-1})(\vt-\vt_0^\eta) +\ep \log \vt + \lambda \vt^{\frac B2} \log\vt\big] \psi \, \dS\\ 
+ \sum_{k=1}^n c_{vk} \intO{\Big[\vt \sum_{l=1}^n Y_k\widehat{D}_{kl}\Grad Y_l \cdot \Grad \psi + \vt (\ep(\vr+1)Y_k +\lambda) \frac{\Grad Y_k}{Y_k}\cdot \Grad \psi\Big]} \\
 = \intO{\vr \vc{f} \cdot \vu \psi} 
+  \frac{\delta}{\beta-1} \intO{(\ep \beta \Ov{\vr} \vr^{\beta-1}\psi + \vr^\beta \vu \cdot \Grad \psi - \ep \beta \vr^\beta \psi)} \\
+\delta \intO{(2\ep  \Ov{\vr} \vr\psi + \vr^2 \vu \cdot \Grad \psi - 2\ep  \vr^2 \psi)} 
\end{multline}
for all $\psi \in C^\infty(\Ov{\Omega})$.
Now it is easy to pass to the limit in \eqref{tot_en_aprox}, similarly as in \cite{GPZ}. The limit passage in the other equalities (continuity equation, momentum equation and the species balance) is easy to perform.

On the level of entropy inequality this limit passage does not entail any additional difficulties with respect to the previous limit passage,
since we have all the previous estimates. Therefore we pass to the limit directly and get inequality of the type \eqref{entropy_aprox}, where we have inequality instead of equality and we remove all indices $\eta$.

\subsection{Limit passage $\lambda \to 0$}
Here we still dispose of estimates \eqref{est_1}. Note, however, that the estimate of $\vu$ in $W^{1,2}$ uniformly in $\lambda$ does not follow from the kinetic energy balance (which is not anymore available) but from the entropy inequality. Furthermore, we loose the uniform control of $\log Y_k$ and $Y_k$ in $W^{1,2}$. Nonetheless, see \cite[Formula (6.12)]{GPZ}, we can verify that  
\begin{equation} \label{est_sigma_y}
\|\sumkN \nabla Y_k\|_2 +  \|(\sumkN Y_k)-1\|_6 \leq C(\lambda) \sim \sqrt{\lambda}\to 0 \quad \textrm{for} \quad \lambda \to 0.
\end{equation} 
This bound, together with \eqref{est_1}, implies
\begin{equation} \label{est_yk}
\sumkN \|\nabla Y_k\|_{\frac{12}{7}} \leq C
\end{equation} 
with $C$ independent of $\lambda$. 
The above estimates combined with \eqref{est_1} allow to pass to the limit in the continuity, momentum, species and total energy balances. We have \\
$\bullet$ the approximate continuity equation
\begin{equation} \label{cont_aprox_2}
\varepsilon \vr+\Div (\vr \vu) = \varepsilon\Delta \vr+\ep \Ov{\vr}, \qquad 
\Grad\vr\cdot\vc{n}|_{\partial \Omega}=0
\end{equation}
$\bullet$ the weak formulation of the approximate momentum equation
\begin{equation} \label{mom_aprox_2}
\begin{array}{c}
\displaystyle
\int_\Omega\Big(\frac{1}{2}\vr\vu\cdot\Grad\vu\cdot\vcg{\varphi}-\frac{1}{2}\vr\left(\vu\otimes\vu\right):\Grad\vcg{\varphi}-\tn{S}:\Grad\vcg{\varphi}\Big)\dx\\
\displaystyle 
-\int_\Omega (\pi+\delta\vr^{\beta}+\delta\vr^{2})\Div\vcg{\varphi} \dx
=\int_\Omega\vr \vc{f}\cdot\vcg{\varphi} \dx
\end{array}
\end{equation}
for all $\vcg{\varphi}\in C^\infty_0(\Omega)$\\
$\bullet$ the weak formulation of the approximate species balance equations  
\begin{equation}\label{aprox_spec_2}
\begin{array}{c}
\displaystyle  \int_\Omega \Big(\ep Y_k\vr\psi -Y_{k}\vr\vu\cdot \Grad \psi +\sum_{l=1}^n Y_k\widehat{D}_{kl}\Grad Y_l \cdot \Grad \psi\Big)\dx \\ 
\displaystyle =  \intO{\Big[\omega_{k} \psi-\varepsilon \vr\Grad Y_k \cdot \Grad \psi+\varepsilon\Div(Y_k\Grad\vr)\psi-\varepsilon\Grad Y_k\cdot \Grad \psi+\varepsilon \Ov{\vr}_k\psi\Big]},
\end{array}
\end{equation}
for all $\psi \in C^\infty(\Ov{\Omega})$ ($k=1,2,\dots,n$)\\
$\bullet$ the weak formulation of the approximate total energy equation
\begin{equation} \label{aprox_tot_en_2}
\begin{array}{c}
\displaystyle -\intO{\Big[\vr e+ \frac 12 \vr |\vu|^2 + (\pi + \delta \vr^\beta + \delta \vr^2)\Big]\vu\cdot \Grad \psi} 
\\
\displaystyle -\intO{\Big(\tn{S} \vu \cdot \Grad \psi + \delta \vt^{-1} \psi\Big)} 
+\intO{\kappa_{\delta}{\frac{\varepsilon+\vt}{\vt}}\Grad\vt\cdot\Grad \psi} \\
\displaystyle  + \int_{\partial \Omega}\big[(L+\delta \vt^{B-1})(\vt-\vt_0) +\ep \log \vt\big] \psi \, \dS \\ 
\displaystyle + \intO{\Big[\vt\sum_{k,l=1}^n c_{vk} Y_k  \widehat{D}_{kl}\Grad Y_l \cdot \Grad \psi + \vt \sum_{k=1}^n \ep(\vr+1) c_{vk} \Grad Y_k \cdot \Grad \psi\Big]} \\
\displaystyle  = \intO{\vr \vc{f} \cdot \vu \psi} + \frac{\delta}{\beta-1} \intO{(\varepsilon \beta \Ov{\vr} \vr^{\beta-1}\psi + \vr^\beta \vu \cdot \Grad \psi - \varepsilon \beta \vr^\beta \psi)} \\
\displaystyle+\delta \intO{(2\ep  \Ov{\vr} \vr\psi + \vr^2 \vu \cdot \Grad \psi - 2\ep  \vr^2 \psi)} 
\end{array}
\end{equation}
for all $\psi \in C^\infty(\Ov{\Omega})$

Next we consider the limit passage in the entropy inequality.
The terms on the left-hand side (lhs) can be treated as in the previous limit passage. 
We only have to pay attention to the terms containing $\log \vr$  and $\log Y_k$. In the former, we use
the approximation $s_k^\lambda$. Namely, we have
$$ 
\int_{\Omega} \psi \vr \vu\cdot \Big[\Big(\sumkN Y_k\Big)\nabla \log(\vr+\sqrt{\lambda})-\nabla \log \vr\Big]\dx=
\int_{\Omega} \psi \vu \cdot \nabla \vr \Big[\Big(\sumkN Y_k\Big)\frac{\vr}{\vr+\sqrt{\lambda}}-1\Big]\dx \to 0.
$$
The next term we should look at is the last term on the rhs. 
After passage with $\lambda$ the part with $\varepsilon$ will vanish due to \eqref{est_sigma_y} and \eqref{est_1}.
Thus it is enough to treat the second term which reads 
$$
\int_{\Omega} \psi \frac{\lambda}{\vr+\sqrt{\lambda}}\nabla \vr \cdot \sumkN \frac{\nabla Y_k}{Y_k} 
$$
and tends to $0$ as $\frac{\lambda}{\vr+\sqrt{\lambda}}\to 0$ and the rest is bounded in $L_1$.
To show convergence of the second last term on the rhs, we use the bound of $\frac{\nabla Y_k}{Y_k}$ and $\log Y_k$  in $L^2$ from \eqref{est_1} (and, indeed, also other bounds coming from there). Note that it is exactly here, where we need the $\sqrt{\lambda}$ instead of $\lambda$ in \eqref{spec_aprox} to ensure that 
the $\lambda$ part of this term converges to zero as $\lambda \to 0^+$.
The part with $\varepsilon$ and $\log Y_k$ is also complicated, as we miss any estimate of $\log Y_k$ which does not blow up when $\lambda \to 0^+$. To this reason, we write
\begin{equation} \label{corr_byparts}
\begin{array}{c}
\displaystyle \ep \int_\Omega \sumkN (\vr+1) \log Y_k \nabla Y_k \cdot \nabla \psi \dx = - \ep \int_\Omega \sumkN (\vr+1) \nabla Y_k \cdot \nabla \psi \dx \\
\displaystyle - \ep \int_\Omega \sumkN Y_k \log Y_k \nabla \vr \cdot \nabla \psi\dx 
- \ep \int_\Omega \sumkN (\vr+1) Y_k \log Y_k \Delta \psi \dx \\
\displaystyle  + \ep \int_{\partial \Omega} \sumkN (\vr+1) Y_k \log Y_k \nabla \psi \cdot \vc{n} \, \dS.
\end{array}
\end{equation} 
Now it is easy to let $\lambda \to 0^+$ in all terms in \eqref{corr_byparts}.
The remaining terms coming from $s_k^\lambda$ cause no troubles. The term with $\log(\vr+\sqrt{\lambda})$  tends to zero as $\sumkN \nabla Y_k$ goes to zero  faster than $\log \lambda$ blows up; the other term with $\log \vt$ is well defined.

Finally, the form of internal energy and \eqref{est_sigma_y} imply that after passing with $\lambda$
we have
\begin{multline*}
\int_{\Omega}\ep(\Delta \vr +\bar \vr-\vr)(\vr^{\gamma-1}+e+\vt)\frac{\psi}{\vt}\dx
\\=\ep\frac{\gamma}{\gamma-1}\int_{\Omega} \frac{\psi}{\vt} \vr^{\gamma-1}(\Delta \vr +\bar \vr-\vr)\dx
+\ep\int_{\Omega}\psi(\Delta \vr +\bar \vr-\vr)\sumkN c_{pk}Y_k\dx.
\end{multline*}
The second term cancels with the 5th term on the rhs of \eqref{entropy_aprox}.
We can therefore
pass to the limit with $\lambda$ obtaining 
\begin{multline} \label{entropy_aprox_epsilon}
\int_{\Omega} \frac{\psi \tn{S} : \nabla \vu}{\vt}\dx
+\int \kappad\frac{(\varepsilon+\vt)}{\vt}\frac{|\nabla \vt|^2}{\vt^2}\psi \dx
-\int_{\Omega} \psi\sumkN  \omega_k (c_{pk}-c_{vk} \log \vt + \log Y_k)\psi \dx\\
+\int_{\Omega}\frac{\delta \psi}{\vt^2}\dx + \int_{\Omega} \psi \sumkN \sumlN \widehat{D}_{kl}\nabla Y_l \nabla Y_k\dx
+ \int_{\partial \Omega} \frac{\psi}{\vt}(L+\delta \vt^{B-1})\vt_0 \, \dS\\ 
+\int_{\Omega}\frac{\delta \varepsilon (\beta \vr^{\beta-2}+2)|\nabla \vr|^2\psi}{\vt}\dx
+\int_{\Omega}\psi\sumkN\ep(\vr+1)\frac{|\nabla Y_k|^2}{Y_k}\dx \\
+\ep\gamma\int_{\Omega}\frac{\psi}{\vt}|\nabla \vr|^2\dx
+\ep\frac{\gamma}{\gamma-1}\int_{\Omega}\frac{\psi}{\vt}\vr^\gamma\dx\\
\leq \int \frac{\kappad(\varepsilon+\vt)\nabla \vt \cdot \nabla \psi}{\vt^2} \dx
-\int_{\Omega} \vr s \vu \cdot \nabla \psi \dx
-\int_{\Omega} \sumkN (c_{vk} \log \vt-\log Y_k) \widehat{\vc{F}}_k \cdot\nabla \psi \dx \\ 
+ \int_{\partial \Omega} \frac{\psi}{\vt} \big((L+\delta\vt^{B-1})\vt +\ep \log \vt) \, \dS 
-\varepsilon \int_{\Omega} \sumkN Y_k \nabla \vr \cdot \nabla \Big(\frac{g_k \psi}{\vt}\Big) \dx\\
-\ep\gamma\int_{\Omega}\vr^{\gamma-1}\nabla\vr\cdot\frac{\nabla\psi}{\vt}\dx
+\ep\frac{\gamma}{\gamma-1}\int_{\Omega}\psi\vr^{\gamma-1}\nabla\vr\cdot\frac{\nabla\vt}{\vt^2}\dx
+\ep\frac{\gamma}{\gamma-1}\int_{\Omega}\frac{\psi}{\vt}\bar\vr\vr^{\gamma-1}\dx\\
+\varepsilon \frac{M}{|\Omega|}  \int_{\Omega} \psi \sumkN  (c_{pk}-c_{vk}\log \vt + 1_{\{\vr>1\}} \log \vr)\dx \\
-\varepsilon \int_{\Omega} \sumkN Y_k \vr \frac{g_k \psi}{\vt}\dx
+ \ep \int_\Omega \sumkN (\vr+1) \nabla Y_k \cdot \nabla \psi \dx \\
+ \ep \int_\Omega \sumkN Y_k \log Y_k \nabla \vr \cdot \nabla \psi \dx
+ \ep \int_\Omega \sumkN (\vr+1) Y_k \log Y_k \Delta \psi \dx \\
- \ep \int_{\partial \Omega} \sumkN (\vr+1) Y_k \log Y_k \nabla \psi \cdot \vc{n} \, \dS + \ep \int_\Omega \sumkN (\vr+1) c_{vk}  \log \vt \nabla Y_k \cdot \nabla \psi \dx,
\end{multline}
where we have integrated by parts the term $\int\frac{\psi}{\vt}\vr^{\gamma-1}\Delta\vr\dx$, used the fact that $\log Y_k \leq 0$ for $\lambda =0$, $\log \vr <0$ for $\vr <1$ 
and $\widehat{D}_{kl}$ is defined in \eqref{dkl_aprox}.

\subsection{Limit passage $\ep \to 0$} 

First of all, we have the following estimates independent of $\ep$:
\begin{multline} \label{est_2}
\sqrt{\ep}\Big(\Big\|\frac{|\nabla Y_k|}{\sqrt{Y_k}}\Big\|_2+ \Big\|\frac{\nabla \vr}{\vt^{1/2}}\Big\|_2 + \|\nabla \vr\|_2\Big) + \sumkN(\|Y_k\|_{1,2} + \|Y_k\|_{\infty})   + \|\Grad \vt^{\frac B2}\|_2 + \|\vt\|_{B,\partial \Omega} \\
 + \|\vt\|_{3m} + \|\vt^{-2}\|_1 
+\Big\|\frac{\nabla \vt}{\vt^{3/2}}\Big\|_2 +\|\vt^{-1}\|_{1,\partial \Omega} + \|\vu\|_{1,2} 
\leq C \Big(1+ \Big|\intO{\vr \vc{f}\cdot \vu}\Big|\Big).
\end{multline} 
These estimates follow from the entropy inequality and the total energy balance, both with the test function $\psi \equiv 1$, and the continuity equation.
At this stage we cannot dispose of the estimates on the density (except the $L^1$ bound due to
given mass) since they depend on $\varepsilon$. We have to show some estimates of the density which will imply that the rhs of \eqref{est_2} can be controlled.

Note that the momentum equation is in fact the same as in the case of the 
compressible Navier--Stokes--Fourier system studied in  \cite{NoPo1}, so we may apply the same 
technique to obtain the so called Bogovskii-type of estimates. 
Following \cite{NoPo1}, we use as test function in \eqref{weak_mom} the function $\vcg{\phi}$, solution to
$$
\displaystyle \Div \vcg{\phi} = \vr^{\frac 23 \beta} - \frac{1}{|\Omega|}\intO {\vr^{\frac 23\beta}}, \qquad
\vcg{\phi}|_{\partial \Omega} = \vc{0}.
$$   
For more information on the Bogovskii operator, we refer the reader to e.g. \cite[Lemma 3.17]{NS}.
In consequence of this testing we may obtain the additional bound on $\vr$, namely
\eqh{ \label{est_rho}
\|\vr\|_{\frac 53 \beta} \leq C,
}
which allows to estimate the rhs of \eqref{est_2}. Now we can proceed with the limit passage. 
Note that the estimates of the density do not imply the compactness of it, however, 
using the DiPerna--Lions renormalization technique applied on the continuity equation 
and the consequences of the effective viscous flux identity, as it is well-known in the 
case of compressible Navier--Stokes(--Fourier) system, we may show the strong convergence 
of the densities in $L^p$ for any $p<\beta$. 
As we have to repeat this proceedure also in the final limit passage we present 
the crucial steps there, referring for more details 
to \cite{NS} or to \cite{NoPo1} in the case of heat--conducting fluid.
Therefore we have after the limit passage $\ep\to 0^+$\\
$\bullet$ the continuity equation
\begin{equation} \label{cont_aprox_3}
\int_\Omega \vr \vu\cdot \nabla \psi= 0
\end{equation}
for all $\psi \in C^\infty(\Ov{\Omega})$ \\
$\bullet$ the weak formulation of the approximate momentum equation
\begin{equation} \label{aprox_mom_3}
\begin{array}{c}
\displaystyle
\int_\Omega\Big(-\vr\left(\vu\otimes\vu\right):\Grad\vcg{\varphi}-\tn{S}:\Grad\vcg{\varphi}\Big)\dx\\
\displaystyle 
-\int_\Omega (\pi+\delta\vr^{\beta}+\delta\vr^{2})\Div\vcg{\varphi} \dx
=\int_\Omega\vr \vc{f}\cdot\vcg{\varphi} \dx
\end{array}
\end{equation}
for all $\vcg{\varphi}\in C^\infty_0(\Omega)$\\
$\bullet$ the weak formulation of the approximate species balance equations  
\begin{equation}\label{aprox_spec_3}
\begin{array}{c}
\displaystyle  \int_\Omega \Big( -Y_{k}\vr\vu\cdot \Grad \psi +\sum_{l=1}^n Y_k D_{kl}\Grad Y_l \cdot \Grad \psi\Big)\dx  =  \intO{\omega_{k} \psi}
\end{array}
\end{equation}
for all $\psi \in C^\infty(\Ov{\Omega})$ ($k=1,2,\dots,n$)\\
$\bullet$ the weak formulation of the approximate total energy equation
\begin{equation} \label{aprox_tot_en_3}
\begin{array}{c}
\displaystyle -\intO{\Big[\vr e+ \frac 12 \vr |\vu|^2 + (\pi + \delta \vr^\beta + \delta \vr^2)\Big]\vu\cdot \Grad \psi} 
\\
\displaystyle -\intO{\Big(\tn{S} \vu \cdot \Grad \psi + \delta \vt^{-1} \psi\Big)} 
+\intO{\kappa_{\delta}\Grad\vt\cdot\Grad \psi} \\
\displaystyle  + \int_{\partial \Omega}\big[(L+\delta \vt^{B-1})(\vt-\vt_0)\big] \psi \, \dS 
\displaystyle + \intO{\vt\sum_{k,l=1}^n c_{vk} Y_k  D_{kl}\Grad Y_l \cdot \Grad \psi} \\
\displaystyle  = \intO{\vr \vc{f} \cdot \vu \psi} + \frac{\delta}{\beta-1} \intO{\vr^\beta \vu \cdot \Grad \psi} 
\displaystyle+\delta \intO{\vr^2 \vu \cdot \Grad \psi} 
\end{array}
\end{equation}
for all $\psi \in C^\infty(\Ov{\Omega})$

Next we deal with the limit passage in the entropy inequality. The lhs does not cause any troubles: we use the weak lower semicontinuity of certain terms or simply cancel some non-negative terms. Most of the terms are easy to treat, the only difficult one is in fact the term
$$
\ep\gamma\int_{\Omega}\frac{\vr^{\gamma-1}}{\vt^2}\nabla\vr\cdot\nabla\vt \psi \dx
$$
which must be controlled by the lhs (in fact, already at the moment when we want to deduce the $\ep$-independent estimates). 
However, using the fact that $\ep \ll \delta$ and $\beta$ is sufficiently high 
we may estimate it by 
$$
\frac 14 \int_{\Omega}\frac{\delta \varepsilon (\beta \vr^{\beta-2}+2)|\nabla \vr|^2\psi}{\vt}\dx + \frac 14 \int \kappad\frac{(\varepsilon+\vt)}{\vt}\frac{|\nabla \vt|^2}{\vt^2}\psi \dx,
$$
in particular by the part $\delta \vt^{-1}$ in $\kappad$. 
The other terms are easy to treat and we end up with 
\begin{multline} \label{entropy_aprox_delta}
\int_{\Omega} \frac{\psi \tn{S} : \nabla \vu}{\vt}\dx
+\int \kappad\frac{|\nabla \vt|^2}{\vt^2}\psi \dx
-\int_{\Omega}\sumkN \omega_k (c_{pk}-c_{vk} \log \vt + \log Y_k)\psi \dx\\
+\int_{\Omega}\frac{\delta \psi}{\vt^2}\dx + \int_{\Omega} \psi \sumkN \sumlN D_{kl}\nabla Y_l \nabla Y_k\dx
+ \int_{\partial \Omega} \frac{\psi}{\vt}(L+\delta \vt^{B-1})\vt_0 \, \dS\\ 
\leq \int \frac{\kappad\nabla \vt \cdot \nabla \psi}{\vt} \dx
-\int_{\Omega} \vr s \vu \cdot \nabla \psi \dx
-\int_{\Omega} \sumkN (c_{vk} \log \vt-\log Y_k) \vc{F}_k \cdot\nabla \psi \dx \\ 
+ \int_{\partial \Omega} (L+\delta\vt^{B-1}) \psi \, \dS. 
\end{multline}

\section{Limit passage $\delta \to 0$}

In the final limit passage we can distinguish three steps. 
The first is in fact a direct application of the method 
from \cite{NoPo1}, where we refer for details. 
In the second step we derive new pressure estimates using the approach from \cite{NoPo2}. 
In fact, we clarify here one estimate in more details, cf. \cite{MPZ_HB}. We can therefore 
pass to the limit in the equations and the entropy inequality, however, we are not able to identify the weak limits in the terms which are non-linear in the density. To this aim, we 
finally show the strong convergence of the density using the techniques developed for compressible 
Navier--Stokes system (which is possible as the momentum and continuity equations are indeed the same).
\subsection{Estimates independent of $\delta$}

Unlike the previous sections, we will denote throughout this section by $(\vrd, \vud,\vtd, \vcYd)$ the solution corresponding to $\delta>0$, while $(\vr,\vu,\vt,\vec{Y})$ will denote the (weak or strong) limits of the corresponding functions when $\delta \to 0^+$. Furthermore, $\vec{Y}=(Y_1,Y_2,\dots,Y_n)^T$, similarly for $\vcYd$.

\subsubsection{Estimates from the entropy inequality}
 
From the total energy balance \eqref{aprox_tot_en_3} tested by a constant function 
we derive
\begin{equation} \label{8.9b}
\|\vtd\|_{1,\partial \Omega} + \delta \|\vtd^B\|_{1,\partial \Omega} \leq C\Big(1+ \Big|\intO{\vrd\vud\cdot \vc{f}}\Big| + \delta \|\vtd^{-1}\|_1\Big).
\end{equation} 
Next, the entropy inequality \eqref{entropy_aprox_delta} with $\psi \equiv C$ yields  
\begin{multline} \label{8.9a}
\|\Grad \vcYd\|_2^2 + \|\Grad \vtd^{\frac m2}\|_2^2 + \|\vud\|_{1,2}^2 + \|\vtd^{-1}\|_{1,\partial \Omega} \\  + \delta \big(\|\Grad \vtd^{\frac B2}\|_2^2 + \|\Grad \vtd^{-\frac 12}\|_2^2 + \|\vtd^{-2}\|_1 +\|\vtd^{B-2}\|_{1,\partial \Omega}\big) \leq C(1+ \delta \|\vtd^{B-1}\|_{1,\partial \Omega}).
\end{multline}
Recall also that we know $0\leq (Y_k)_\delta\leq 1$, $k=1,2,\dots, n$. 
In order to get rid of the $\delta$-dependent terms in the above estimates we apply once again Bogovskii-type estimates,
this time testing the momentum equation by a solution to  
$$
\Div \, \vcg{\phi} = \vrd-\frac{M}{|\Omega|},  \qquad
\vcg{\phi}|_{\partial \Omega} = \vc{0}.
$$
It is an easy matter to verify the bound (see also \cite{NoPo1})
$$
\delta \|\vrd\|_{\beta+1}^{\beta-\frac 32} \leq C.
$$
Applying this estimate to \eqref{8.9a} and \eqref{8.9b} we can get rid of most of $\delta$-terms obtaining
\eq{\label{8.14}
&\|\Grad \vcYd\|_2^2+\|\vcYd\|_{\infty} + \|\Grad \vtd^{\frac m2}\|_2 + \|\vud\|_{1,2} + \|\vtd^{-1}\|_{1,\partial \Omega} \\
&+ \delta (\|\Grad \vtd^{\frac B2}\|_2^2 + \|\Grad \vtd^{-\frac 12}\|_2^2 + \|\vtd^{-2}\|_1 +\|\vtd^{B-2}\|_{1,\partial \Omega}) \leq C
}
and
\eq{\label {vt_3m_0}
\|\vtd\|_{3m} \leq C\Big(1+ \Big|\int_{\Omega}\vrd\vud\cdot\vc{f}\dx\Big| \Big).
}
See also \cite{KrNePo} for similar computations in the case of a more complex dependence of the viscosity on the temperature. 
\subsubsection{Local pressure estimates}

The second step consist in derivation of $\delta$-independent estimates for the density. This is the core estimate
which finally will allow us to get a bound $\gamma>\frac{4}{3}$ for weak solutions and $\gamma>1$ for variational
entropy solutions.  
Here we follow the idea of local pressure estimates introduced in several papers by Plotnikov and Sokolowski (see \cite{PlSo}), Novotn\'y and B\v{r}ezina (\cite{NoBr}) and Frehse, Steinhauer and Weigant (\cite{FrStWe}) and applied to the compressible Navier--Stokes--Fourier system in \cite{NoPo2}; see also \cite{MPZ_HB} for further information.
 
For $b>1$ let us denote  
$$
A = \int_{\Omega}\vrd^b|\vud|^2\dx.
$$
Applying H\"older's inequality to the rhs of \eqref{vt_3m_0} we get 
\begin{equation} \label{vt_3m}
\|\vtd\|_{3m} \leq C(1+A^{\frac{1}{6b-4}}).
\end{equation}
Next we apply once again Bogovskii-type estimate to show
\begin{lem} \label{lem1}
We have for $1<s\leq \frac{3b}{b+2}$, $s\leq \frac{6m}{2+3m}$, $m>\frac 23$ and $b\geq 1$
\begin{equation} \label{3.12}
\intO{\vrd^{s\gamma}}
+ \intO {\vrd^{(s-1)\gamma} \pi(\vrd,\vtd)} + \intO {\big(\vrd |\vud|^2\big)^s} + \delta \intO {\vrd ^{\beta + (s-1)\gamma}} \leq C\big(1+ A^{\frac {4s-3}{3b-2}}\big).
\end{equation}
\end{lem}
\begin{proof} We sketch the main steps referring to \cite{NoPo2} for more details. Testing the momentum 
equation with $\vcg{\phi}$ solving 
$$
\Div \, \vcg{\phi} = \vrd^{(s-1)\gamma}-\frac{1}{|\Omega|}\int_{\Omega} \vrd^{(s-1)\gamma}\dx ,  \qquad
\vcg{\phi}|_{\partial \Omega} = \vc{0}
$$
we obtain 
\begin{equation} \label{bog2}
\begin{array}{c}
\displaystyle \intO {\vrd^{(s-1)\gamma} \pi(\vrd,\vtd)} + \delta \intO {\vrd^{(s-1)\gamma} \big(\vrd^\beta + \vrd^2)} = \frac{1}{|\Omega|} \intO {\pi(\vrd,\vtd)} \intO {\vrd^{(s-1)\gamma}} \\[9pt]
\displaystyle + \frac{\delta}{|\Omega|} \intO{(\vrd^\beta + \vrd^2)}  \intO {\vrd^{(s-1)\gamma}} - \intO{\vrd(\vud\otimes \vud):\nabla \vcg{\phi}} 
+ \intO{\tn{S}(\vtd,\nabla \vud): \nabla \vcg{\phi}}\\
\displaystyle - \intO{\vr \vc{f} \cdot \vcg{\phi}} = I_1 + I_2 + I_3 + I_4 + I_5.
\end{array}
\end{equation}
We have to estimate the rhs. The most restrictive terms are $I_3$, giving the restriction on $s$, and $I_4$ which leads to the other restrictions, especially to $m>\frac 23$. For more details see \cite{NoPo2}.
\end{proof}
    
Now we come to the core of our estimates.
The idea is to test the momentum equation by a cleverly chosen function involving the distance from the boundary 
to find a bound
$$ 
{\rm sup}_{x_0 \in \bar \Omega}\int_{\Omega}\frac{\pi(\vrd,\vtd)}{|x-x_0|^\alpha}\dx \leq C 
$$
for some $\alpha>0$ with $C$ independent from $\delta$. We have to use different test 
functions distinguishing 3 cases: $x_0$ far from the boundary, $x_0$ at the boundary 
and finally $x_0$ close to the boundary. The first two cases are treated in details 
in \cite{NoPo2}, therefore we only recall the results here. The third case is most delicate 
and has not been presented so far and some ideas can be only found in \cite{MPZ_HB}. 

The case of $x_0$ far from the boundary is the easiest. We test the momentum equation 
\eqref{aprox_mom_3} with 
\begin{equation} \label{def_phizero}
\vphi^0(x) = \frac{x-x_0}{|x-x_0|^\alpha} \tau^2
\end{equation}
with $\tau \equiv 1$ in $B_{R_0}(x_0)$, 
$\tau \equiv 0$ outside $B_{2R_0}(x_0)$ with $R_0$ as below, 
$|\nabla \tau| \leq \frac{C}{R_0}$. 
Calculating directly the derivatives of $\varphi$  
we obtain (see \cite[Lemma 3.4]{NoPo2} or \cite{MPZ_HB}): 
\begin{lem}  \label{lem_int}
Let $x_0 \in \Omega$, $R_0 < \frac 13 \mbox{dist}\, (x_0,\partial\Omega)$. Then
\begin{equation} \label{3.14}
\begin{array}{c}
\displaystyle
\int_{B_{R_0}(x_0)} \frac{\pi(\vrd,\vtd)+\delta(\vrd^\beta+\vrd^2)}{|x-x_0|^\alpha} \dx \\
\displaystyle \leq C\big(1+ \|\pi(\vrd,\vtd)\|_1 + \|\vud\|_{1,2}(1+\|\vtd\|_{3m}) + \|\vrd|\vud|^2\|_1\big),
\end{array}
\end{equation}
provided
\begin{equation} \label{3.15}
\alpha < \min\Big\{\frac{3m-2}{2m},1\Big\}.
\end{equation}
\end{lem}
Next we treat the case $x_0 \in \partial\Omega$. This time we use in \eqref{aprox_mom_3}
a test function 
\begin{equation} \label{phi1}
\vcg{\varphi}^1(x) = d(x) \nabla d(x) (d(x) + |x-x_0|^a)^{-\alpha}
\end{equation}  
where $a=\frac{2}{2-\alpha}$ and  $d(x)$ is a function which behaves like ${\rm dist}(x,\partial\Omega)$ near the boundary and it is a $C^2(\Ov{\Omega})$ function. 
It can be shown (see \cite[Lemma 3.5]{NoPo2} or \cite{MPZ_HB}) that
$\vcg{\varphi}^1 \in W^{1,q}_0(\Omega)$ for $1\leq q < \frac{3-\alpha}{\alpha}$ 
and
\begin{equation} \label{prop_phi1}
\begin{array}{c}
\displaystyle \partial_j \vcg{\varphi}_i^1(x) = \frac{d(x) \partial^2_{ij}d(x)}{(d(x)+|x-x_0|^a)^{\alpha}} + \frac{(1-\alpha) d(x) + |x-x_0|^a}{2(d(x)+|x-x_0|^a)^{1+\alpha}} \partial_i d(x) \partial_j d(x) \\[8pt]
\displaystyle + \frac{(1-\alpha)d(x) + |x-x_0|^a}{2(d(x)+ |x-x_0|^a)^{1+\alpha}} (\partial_i d(x) - \mu^i(x))(\partial_j d(x) - \mu^j(x)) \\[8pt]
\displaystyle + \frac{\alpha d(x) [\partial_j d(x) \partial_i (|x-x_0|^a)- \partial_i d(x) \partial_j (|x-x_0|^a)]}{2(d(x)+|x-x_0|^a)^{1+\alpha}} \\[8pt]
\displaystyle - \frac{\alpha^2 d^2(x) \partial_i (|x-x_0|^a) \partial_j (|x-x_0|^a)}{2(d(x)+|x-x_0|^a)^{1+\alpha} \big((1-\alpha)d(x) + |x-x_0|^a\big)},
\end{array}
\end{equation}
where
$$
\mu^i(x) = \alpha d(x) \big((1-\alpha)d(x) + |x-x_0|^a\big)^{-1} \partial_i (|x-x_0|^a),
$$
$i=1,2,3$.
These properties of $\vcg{\varphi}^1$ enable to show the 
estimate (\cite{NoPo2}, Lemma 3.6):
\begin{lem} \label{lem_bdr}
Under the assumptions above, we have for $\alpha < \frac{9m-6}{9m-2}$, $x_0 \in \partial \Omega$ and $R_0$ sufficiently small (uniformly with respect to $x_0$) 
\begin{equation} \label{3.20}
\int_{B_{R_0}(x_0)\cap \Omega} \frac{\pi(\vrd,\vtd)+\delta(\vrd^\beta+\vrd^2)}{|x-x_0|^\alpha} \dx \leq 
C\big(1+ \|\pi(\vrd,\vtd)\|_1 + (1+\|\vtd\|_{3m})\|\vud\|_{1,2} + \|\vrd |\vud|^2\|_1\big).
\end{equation}
\end{lem} 
Now we come to the most delicate part of the estimate. 
Notice that in Lemma \ref{lem_int} the ball is separated from the boundary,
therefore we have to treat separately the case of $x_0 \in \Omega$ which 
is close to the boundary. 
This gap was not commented in the original papers, here we fill it using  
a carefully chosen test function vanishing at the boundary, which enables 
us to reach with the ball up to the boundary. Precisely, we show the following 
\begin{lem} \label{lem_close}
Assume that $x_0 \in \Omega$ is such that ${\rm dist}\{x_0,\partial\Omega\}=5\ep$
for some $0<\ep\ll 1$ and $\alpha<\frac{9m-6}{9m-2}$. Then 
\begin{equation} \label{est_close}
\intO{\frac{\pi(\vrd,\vtd)+\delta(\vrd^\beta+\vrd^2)}{|x-x_0|^\alpha}}
\leq C\big(1+ \|\pi(\vrd,\vtd)\|_1 + \|\vud\|_{1,2}(1+\|\vtd\|_{3m}) + \|\vrd|\vud|^2\|_1\big).
\end{equation}
\end{lem}
\begin{proof}
We use again the function $\vcg{\varphi}^1$ defined in \eqref{phi1}. 
From \eqref{prop_phi1} we see that  
\begin{equation} \label{e1}
\intO{\vrd(\vud\otimes\vud):\nabla\vphi^1}\geq 
C_1 \intO{ \frac{\vrd(\vud\cdot\nabla d)^2}{(d(x)+|x-x_0|^a)^\alpha} }
-C_2\intO{ \vrd|\vud|^2 }
\end{equation}
and
\begin{multline} \label{e2}
\intO{ [\pi(\vrd,\vtd)+\delta(\vrd^\beta+\vrd^2)]\Div\, \vphi^1 } \geq 
C_1 \intO{ \frac{\pi(\vrd,\vtd)+\delta(\vrd^\beta+\vrd^2)}{(d(x)+|x-x_0|^a)^\alpha} } \\
-C_2 \intO{ \big(\pi(\vrd,\vtd)+\delta(\vrd^\beta+\vrd^2)\big) }.
\end{multline} 
The form of $\nabla \vphi^1$ in \eqref{prop_phi1} imply for $q<\frac{3-\alpha}{\alpha}$ that
$\|\vphi^1\|_{1,q} \leq C$ independently of the distance from the boundary.
However, we have 
$$
\frac{1}{(d(x)+|x-x_0|^a)^\alpha} \geq \frac{C}{|x-x_0|^\alpha}
$$
only for $x\in\Omega\setminus B_\ep(x_0)$.
Therefore \eqref{e1} and \eqref{e2} does not provide estimate for $\frac{\pi}{|x-x_0|^\alpha}$
in $B_\ep(x_0)$ where we need an additional estimate. 
To this end we introduce additional function which behaves 
like $\vphi^0$ defined in \eqref{def_phizero}, but additionally vanishes on the boundary. 
To combine these requirements we define it in a following way:
\begin{equation}
\vphi^2(x)=\left\{ 
\begin{array}{lr}
\frac{x-x_0}{|x-x_0|^\alpha}\Big(1-\frac{1}{2^\frac{\alpha}{2}}\Big), & |x-x_0|<\ep,\\
(x-x_0)\Big(\frac{1}{|x-x_0|^\frac{\alpha}{2}}-\frac{1}{(|x-x_0|+\ep)^\frac{\alpha}{2}}\Big), & |x-x_0|>\ep,d(x)>\ep,\\ 
(x-x_0)\Big(\frac{1}{|x-x_0|^\frac{\alpha}{2}}-\frac{1}{(|x-x_0|+d(x))^\frac{\alpha}{2}}\Big), & |x-x_0|>\ep,d(x)\leq\ep.
\end{array} \right.
\end{equation}
First of all, we easily verify that 
$$
\vphi^2 \in W_0^{1,q}(\Omega) \quad \textrm{for all} \quad 1 \leq q < \frac{3}{\alpha}
$$ 
with the norm bounded independently of $\ep$. Indeed, the singularity in $\vphi^2$
and its derivatives appears only in $B_\ep(x_0)$, where we have 
\begin{equation} \label{phi2}
\nabla \vphi^2 \sim \nabla \frac{x-x_0}{|x-x_0|^\alpha} \sim \frac{1}{|x-x_0|^\alpha},
\end{equation} 
which yields the above limitation on $q$. 
Now we can verify that
\begin{multline} \label{e3}
\intO{\vrd(\vud\otimes\vud):\nabla\vphi^2}\geq 
K_1 \int_{B_\ep(x_0)}\frac{\vrd|\vud|^2}{|x-x_0|^\alpha}\dx\\
-K_2\int_{\{ d(x)<\ep \}}\frac{\vrd(\vud\cdot\nabla d)^2}{(d(x)+|x-x_0|^a)^\alpha}\dx
-K_3\intO{\vrd|\vud|^2}  
\end{multline}
and
\begin{multline} \label{e4}
\intO{\big(\pi(\vrd,\vtd)+\delta(\vrd^\beta+\vrd^2)\big)\Div\,\vphi^2}\geq
K_1 \int_{B_\ep(x_0)}\frac{\pi(\vrd,\vtd)+\delta(\vrd^\beta+\vrd^2)}{|x-x_0|^\alpha}\dx\\
-K_3\intO{\big(\pi(\vrd,\vtd)+\delta(\vrd^\beta+\vrd^2)\big)}.
\end{multline}
Notice that the first term on the rhs of \eqref{e4} is exactly the one which we were 
missing (the estimate outside $B_\ep(x_0)$ is given by \eqref{e2}). After all these 
considerations we can test \eqref{aprox_mom_3} with 
$$
\vphi = K\vphi^1+\vphi^2
$$
where $K$ is a sufficiently large constant.  
Then the first term on the rhs of \eqref{e1} which has a good sign compensates 
the second term on the rhs of \eqref{e3} and we conclude \eqref{est_close}
provided $\alpha<\frac{9m-6}{9m-2}$ which completes the proof.
\end{proof}

Combining Lemmas \ref{lem_int}--\ref{lem_close} we conclude 
\begin{prop}
Assume $\alpha<\frac{9m-6}{9m-2}$. Then
\begin{multline} \label{pres_final}
{\rm sup}_{x_0\in\overline\Omega}\intO{\frac{\pi(\vrd,\vtd)+\delta(\vrd^\beta+\vrd^2)}{|x-x_0|^\alpha}}\\
\leq C\big(1+ \|\pi(\vrd,\vtd)\|_1 + \|\vud\|_{1,2}(1+\|\vtd\|_{3m}) + \|\vrd|\vud|^2\|_1\big).
\end{multline}
\end{prop}
Using the above pressure estimate we show
\begin{lem} \label{lem_A}
Let $1\leq b <\gamma$, $\alpha < \frac{9m-6}{9m-2}$ and $\alpha > \frac{3b-2\gamma}{b}$. Then
\begin{equation} \label{est_A}
A = \intO{\vrd^b |\vud|^2}  \leq C \|\vud\|_{1,2}^2 \big(1+ \|\pi(\vrd,\vtd)\|_1 +  \|\vud\|_{1,2}(1+\|\vtd\|_{3m}) + \|\vrd |\vud|^2\|_{1}\big)^{\frac{b}{\gamma}}.
\end{equation}
\end{lem}
This is exactly Lemma 3.7 from \cite{NoPo2}, however we sketch the proof here to show 
the application of \eqref{pres_final} which is not evident.  
First using interpolation inequality we show 
$$ 
\intO {\frac{\vrd^b}{|x-x_0|}} \leq
C \Big(1+ \|\pi(\vrd,\vtd)\|_1 + \|\vud\|_{1,2} (1+ \|\vtd\|_{3m}) + \|\vrd |\vud|^2\|_1 \Big)^{\frac b\gamma}.
$$
Next we introduce $h$ as a solution to $\Delta h = \vrd^b$, $h|_{\partial \Omega}=0$,
and represent it with the Green function to obtain 
\begin{equation*} \label{3.26}
\|h\|_\infty \leq C \sup_{x_0 \in \Omega} \intO {\frac{\vrd^b(x)}{|x-x_0|}}.
\end{equation*}
The definition of $h$ yields 
\begin{equation*} \label{3.27}
A \leq C \|\nabla\vud\|_2 \Big(\intO{|\vud|^2 |\nabla h|^2}\Big)^{\frac 12},
\end{equation*}
and integrating by parts the last integral we get \eqref{est_A}.   
We are now ready to show the following 
\begin{lem} \label{l 3.8}
Let $\gamma>1$, $m>\frac 23$ and $m > \frac{2}{9} \frac \gamma{\gamma-1}$. Then there exists $s>1$ such that $\vrd$ is bounded in $L^{s\gamma}(\Omega)$ and $\pi(\vrd,\vtd)$, $\vrd |\vud|$ and $\vrd |\vud|^2$ are bounded in $L^s(\Omega)$. Moreover, if $\gamma >\frac 43$ and $m>1$ for $\gamma \geq \frac{12}{7}$ and $m>\frac{2\gamma}{3(3\gamma-4)}$ for $\gamma \in (\frac 43,\frac{12}{7})$, we can take $s>\frac 65$.
\end{lem}
\begin{proof}
Interpolation inequality yields   
$$
\|\pi(\vrd,\vtd)\|_1 \leq C \Big(\Big(\intO {\vrd^{s\gamma}}\Big)^{\frac 1s}  +
 \Big(\intO{\vrd^{(s-1)\gamma}\pi(\vrd,\vtd)}\Big)^{\frac{1}{(s-1)\gamma +1}} \times \Big(\intO{\vtd}\Big)^{\frac{(s-1)\gamma}{(s-1)\gamma +1}}\Big).
$$
Therefore, combining Lemmas \ref{lem1} and \ref{lem_A} and applying \eqref{vt_3m}  we show
\begin{equation*} \label{3.29}
A \leq C \Big(1+ A^{\frac{4s-3}{3b-2} \frac 1s} + A^{\frac 1 {6b-4} (1+ \frac{8s-7}{(s-1)\gamma +1})} \Big)^{\frac b\gamma}.
\end{equation*}
In order to get the statement of the Lemma we need 
\begin{equation*} \label{3.30}
\frac{4s-3}{s} \frac{1}{3b-2} \frac{b}{\gamma} <1 \qquad \mbox{ and }  \qquad \frac 1 {6b-4} \Big(1+ \frac{8s-7}{(s-1)\gamma +1}\Big) \frac{b}{\gamma} <1
\end{equation*}
for a certain $s>1$ and $1\leq b <\gamma$. Collecting these and other assumptions 
from this section we get the statement for $\pi$ and $\vrd\vud$ and the result 
for $\vrd|\vud|^2$ follows from  
$$
\|\vrd \vud\|_s \leq \|\vrd\|_s^\frac 12 \|\vrd |\vud|^2\|_s^{\frac 12}.
$$
If we require $s>\frac 65$, we get more restrictions, see \cite{NoPo2} or \cite{MPZ_HB}.
\end{proof}
In order to pass to the limit in the total energy balance, we have to show that
\begin{equation} \label{3.31a}
\lim_{\delta \to 0^+} \delta \|\vrd\|_{\frac 65 \beta}^\beta = 0.
\end{equation}
To this end we use the Bogovskii-type estimates of the momentum equation \eqref{aprox_mom_3}
with  
$\vrd^{\frac{1}{5}\beta+\eta}$, $\eta>0$. Assuming additionally that 
$$
\vrd|\vud|^2 \quad \textrm{is bounded in some} \; L^q(\Omega), \; q>\frac 65,
$$  
we deduce (see \cite{NoPo2} or \cite{MPZ_HB} for details)  
$$
\delta \|\vrd\|_{\frac 65 \beta + \eta}^\beta \leq C
$$
for some $\eta >0$ which yields (\ref{3.31a}) due to interpolation of $L^{\frac 65 \beta}(\Omega)$ 
between $L^1(\Omega) $ and $L^{\frac 65 \beta + \eta}(\Omega)$.\\

\subsection{Limit passage}

\subsubsection{Limit passage based on a priori estimates}

Collecting the estimates obtained so far we have the following convergences
\begin{equation} \label{conv_delta}
\begin{array}{lll}
\displaystyle \vud \rightharpoonup \vu & \mbox{ in } W^{1,2}_0(\Omega), & \vud \to \vu \quad \mbox{ in } L^q(\Omega),\, q<6, \\
\vrd \rightharpoonup \vr & \displaystyle \mbox{ in } L^{s\gamma}(\Omega),  \, \gamma > 1, \, m> \frac 23, \displaystyle m>\frac{2}{9} \frac{\gamma}{\gamma-1},  \\
\vtd \rightharpoonup \vt & \displaystyle \mbox{ in } W^{1,r}(\Omega),\,  r = \min \Big\{2,\frac{3m}{m+1}\Big\}
, & \\
\vtd \to \vt & \mbox{ in } L^q(\Omega),\, q<3m, &  \vtd \to \vt \quad \mbox{ in } L^q(\partial \Omega),\, q< 2m,\\
\displaystyle \vcYd \rightharpoonup \vec{Y} & \mbox{ in } W^{1,2}(\Omega), & \vcYd \to \vec{Y} \quad \mbox{ in } L^q(\Omega),\,q<\infty. 
\end{array}
\end{equation}  
These allow to pass to the limit in the continuity equation, momentum equation,
species balance equations and entropy inequality to obtain 
\eq{\label{cont_lim} 
\intO{\vr \vu \cdot \Grad \psi} = 0
}
for all $\psi \in C^\infty(\Ov{\Omega})$, 
\begin{equation}\label{weak_mom_lim}
-\intOB{\vr\left(\vu\otimes\vu\right):\Grad\vcg{\phi}+\tn{S}:\Grad\vcg{\phi}}-\intO{\big(\vr\vt+\Ov{\vr^\gamma}\big)\Div\, \vcg{\phi}}=\intO{\vr \vc{f}\cdot\vcg{\phi}}
\end{equation}
for all $\vcg{\phi}\in C^\infty_0(\Omega)$,
\eq{\label{8.19}
- \intO{Y_{k}\vr\vu\cdot \Grad \psi} +\intO{Y_k \sum_{l=1}^n D_{kl}\Grad Y_l \cdot \Grad \psi}  
=  \intO{\vw_{k} \psi}
}
for all $\psi \in C^\infty(\Ov{\Omega})$, $k=1,2,\dots, n$,
and 
\begin{multline} \label{entropy_aprox_lim}
\int_{\Omega} \frac{\psi \tn {S} : \nabla \vu}{\vt}\dx
+\int \kappa \frac{|\nabla \vt|^2}{\vt^2}\psi \dx
-\int_{\Omega}\sumkN \psi(c_{pk}-c_{vk}\log\vt+\log Y_k)\omega_k\dx
\\
+\int_{\Omega}\psi \sum_{k,l=1}^n D_{kl}\nabla Y_k \nabla Y_l\dx
+ \int_{\partial \Omega} \frac{\psi}{\vt}L\vt_0 \, \dS \leq 
\int \frac{\kappa \nabla \vt \cdot \nabla \psi}{\vt} \dx
-\int_{\Omega} \overline{\vr s} \vu \cdot \nabla \psi \dx\\
-\int_{\Omega} \sumkN \vc{F}_k \cdot ( c_{vk} \log \vt-\log Y_k) \nabla \psi \dx  
+ \int_{\partial \Omega} \psi L \, \dS 
\end{multline}
for all non-negative $\psi \in C^\infty(\Ov{\Omega})$.
In order to pass in the species equations we need to assume 
$D_{kl}(\vt,\cdot) \leq C(1+ \vt^a)$ for $a <\frac{3m}{2}$,  
no further restrictions on $\gamma$ are needed. 

However, in order to pass in the total energy balance we need $s>\frac 65$ in \eqref{conv_delta}. 
This requirement combined with other assumptions from this section yields 
(see Lemma \ref{l 3.8})
$\gamma>\frac{4}{3}$ and
\begin{equation} \label{4.5}
\begin{array}{c}
\displaystyle m>1 \qquad \mbox{ for } \qquad \gamma \geq \frac{12}{7}, \\[8pt]
\displaystyle m> \frac{2\gamma}{3(3\gamma-4)} \qquad \mbox{ for } \qquad \gamma \in \Big(\frac 43,\frac {12}{7}\Big).
\end{array}
\end{equation}
Furthermore, we also need $a<\frac{3m-2}{2}$.
Under these restrictions we can pass to the limit also in the total energy balance to get
\eq{\label{8.20}
&-\intO{\Big[\vr\vt\sumkN c_{vk}Y_k+ \frac 12 \vr |\vu|^2 + \vr\vt +\frac{\gamma}{\gamma-1}\Ov{\vr^\gamma }\Big]\vu\cdot \Grad \psi}  
-\intO{\vS \vu \cdot \Grad \psi} \\
&+\intO{\kappa\Grad\vt\cdot\Grad \psi} + \int_{\partial \Omega}L(\vt-\vt_0)  \psi \ {\rm d}S 
\\
&+ \intO{\vt\sum_{k,l=1}^n c_{vk} Y_k   D_{kl}\Grad Y_l \cdot \Grad \psi} 
 = \intO{\vr \vc{f} \cdot \vu \psi}.
}

\subsubsection{Strong convergence of the density}
In order to finish the proof of Theorem \ref{t1}  we have to get rid of the 
weak limits in \eqref{cont_lim}--\eqref{8.20} denoted by bars. For this purpose we need to show that 
$$
\vrd \to \vr \quad \textrm{strongly in} \quad L^1(\Omega).
$$
Here we apply the techniques developed for compressible Navier--Stokes system 
which consist in testing the momentum equation with appropriately chosen test function 
leading to so called effective viscous flux identity. As the momentum equation 
is in our case essentially the same we can repeat this approach.
We skip the details as this is already standard in the theory of compressible flows, 
however for the sake of completeness we recall the main steps.  

\noindent {\bf Step 1. Effective viscous flux identity.}
Consider
$$
T_k(z) = k T\Big(\frac{z}{k}\Big), \qquad
T(z) = \left\{ \begin{array}{c}
z \mbox{ for } 0\leq z\leq 1, \\
\mbox{ concave on } (0,\infty), \\
2 \mbox{ for } z\geq 3.
\end{array}
\right.
$$  
Using as a test function $\zeta(x)\nabla\Delta^{-1}(1_\Omega T_k(\vrd))$ 
in the approximate momentum equation \eqref{aprox_mom_3}
and $\zeta(x)\nabla\Delta^{-1}(1_\Omega \overline{T_k(\vr)})$ in its limit version \eqref{weak_mom_lim}
with $\zeta(x)\in C^\infty_0(\Omega)$ we get the identity  
(for the proof see \cite{NoPo1}, Lemma 12 with $T_k(\vr)$ instead of $\vr$): 
\begin{equation} \label{vf1}
\begin{array}{c}
\displaystyle
\lim_{\delta \to 0^+} \intO {\zeta(x) \Big(\pi(\vrd,\vtd) T_k(\vrd) - \tn {S}(\vtd,\nabla\vud):{\cal R}[1_\Omega T_k(\vrd)]\Big) } \\
\displaystyle = \intO {\zeta(x) \Big(\overline{(\pi(\vr,\vt)} \ \overline{T_k(\vr)} -\tn{S}(\vt,\nabla\vu):{\cal R}[1_\Omega \overline{T_k(\vr)}]\Big)} \\
\displaystyle + \lim_{\delta \to 0^+} \intO {\zeta(x) \Big(T_k(\vrd) \vud \cdot {\cal R} [1_\Omega \vrd \vud] - \vrd (\vud\otimes \vud) : {\cal R}[1_\Omega T_k(\vrd)]\Big)} \\[8pt]
\displaystyle - \intO {\zeta(x) \Big(\overline{T_k(\vr)} \vu \cdot {\cal R}
[1_\Omega \vr \vu] - \vr (\vu\otimes \vu) : {\cal R}[1_\Omega
\overline{T_k(\vr)}]\Big)},
\end{array}
\end{equation}
where ${\cal R}$ denotes the double Riesz operator, $({\cal
R}[v])_{ij} = (\nabla\otimes\nabla\Delta^{-1})_{ij} v = {\cal
F}^{-1}\Big[\frac{\xi_i \xi_j}{|\xi|^2} {\cal F}(v)(\xi)\Big]$ with ${\cal F}$ the Fourier transform, and we used that $\Div (\vrd\vud) = \Div(\vr\vu)=0$. 
We recall some auxiliary results we will apply. 
The first one is (see \cite[Theorem 10.27]{FeNoB})
\begin{lem}[Commutators I] \label{lem_com1}
\medskip
Let $\vc{U}_\delta \rightharpoonup \vc{U}$ in
$L^p(\R^3)$, $v_\delta \rightharpoonup v$ in
$L^q(\R^3)$, where
$$
\frac 1p + \frac 1q = \frac 1s <1.
$$
Then
$$
v_\delta {\cal R}[\vc{U}_\delta] - {\cal
R}[{v}_\delta] \vc{U}_\delta \rightharpoonup v {\cal
R}[\vc{U}] - {\cal R}[{v}] \vc{U}
$$
in $L^s(\R^3)$.
\end{lem}
\noindent 
The second is (see Theorem 10.28 in \cite{FeNoB}) 
\begin{lem}[Commutators II] \label{lem_com2}
\medskip
Let $w \in W^{1,r}(\R^3)$, $\vc{z} \in L^p(\R^3)$, $1<r<3$,
$1<p<\infty$, $\frac 1r + \frac 1p -\frac 13 <\frac 1s <1$. Then
for all such $s$ we have
$$
\|{\cal R}[w\vc{z}] - w{\cal R}[\vc{z}] \|_{a,s,\R^3} \leq C
\|w\|_{1,r,\R^3} \|\vc{z}\|_{p,\R^3},
$$
where $\frac a 3 = \frac 1s + \frac 13 -\frac 1p - \frac 1r$. Here, $\|\cdot\|_{a,s, \R^3}$ denotes the norm in
the Sobolev--Slobodetskii space $W^{a,s}(\R^3)$.
\end{lem}
\noindent
Finally we have (\cite[Lemma 6]{JeNoPo}):
\begin{lem} \label{lem_JNP}
Let $\Omega$ bounded, $f_\delta \to f$ in $L^1(\Omega)$, $g_\delta \rightharpoonup g$ in $L^1(\Omega)$ and 
$f_\delta g_\delta \rightharpoonup h$ in $L^1(\Omega)$. Then $h=fg$.
\end{lem}
\noindent
We have the following identity  
\begin{equation} \label{ir}
\vr(\vu \otimes \vu):{\cal R}[T_k(\vr)]=\sum_{i,j=1}^3u_i({\cal R}[T_k(\vr)])_{ij}\vr u_j),
\end{equation}
which is uniformly bounded in $L^p$ for $1\leq p\leq s$. 
Moreover, we have 
$$
\vrd\vud\rightharpoonup \vr\vu, \quad \vrd\vud \otimes \vud\rightharpoonup \vr\vu\otimes\vu \quad 
\textrm{in} \; L^1.
$$
Therefore, applying Lemma \ref{lem_com1} with
$$
\begin{array}{c}
\displaystyle
v_\delta = T_k(\vrd) \rightharpoonup \overline{T_k(\vr)} \qquad \mbox{ in } L^{q}(\R^3), \, q<\infty \mbox{ arbitrary } \\
\displaystyle \vc{U}_\delta = \vrd \vud \rightharpoonup \vr
\vu \qquad \mbox{ in } L^{p}(\R^3), \, \mbox{ for certain } p>1,
\end{array}
$$
and Lemma \ref{lem_JNP} with 
$$
f_\delta = T_k(\vrd){\cal R}[1_\Omega \vrd \vud]-\vrd\vud{\cal R}[T_k(\vrd)], \quad
g_\delta=\zeta \vud 
$$
we obtain 
$$
\begin{array}{c}
\displaystyle
\intO {\zeta(x) \vud \cdot \Big(T_k(\vrd)  {\cal R} [1_\Omega \vrd
\vud] - \vrd {\cal R}[1_\Omega T_k(\vrd)] \vud  \Big)} \\
\displaystyle \to \intO
{\zeta(x) \vu \cdot \Big(\overline{T_k(\vr)}  {\cal R} [1_\Omega \vr \vu] - \vr
{\cal R}[1_\Omega \overline{T_k(\vr)}] \vu  \Big)}.
\end{array}
$$
This convergence in view of \eqref{vf1} and \eqref{ir} gives 
\begin{equation} \label{vf2}
\begin{array}{c}
\lim_{\delta \to 0^+} \intO{\zeta(x) \Big(\pi(\vrd,\vtd) T_k(\vrd) - \tn{S}(\vtd,\nabla\vud):{\cal R}[1_\Omega T_k(\vrd)]\Big)} \\
\displaystyle 
= \intO{\zeta(x) \Big(\overline{(p(\vr,\vt)} \, \overline{T_k(\vr)} -\tn{S}(\vt,\nabla\vu):{\cal R}[1_\Omega \overline{T_k(\vr)}]\Big)}. 
\end{array}
\end{equation}
Next we can write 
\begin{equation} \label{vf3}
\begin{array}{c}
\displaystyle \intO {\zeta(x) \overline{\tn{S}(\vt,\nabla\vu) : {\cal R}[1_\Omega T_k(\vr)]}} = \lim_{\delta \to 0^+} \intO {\zeta(x) \Big(\frac 43 \mu(\vtd)+ \xi(\vtd)\Big) \Div \, \vud T_k(\vrd)} \\
\displaystyle + \lim_{\delta \to 0^+} \intO {T_k(\vtd) \Big({\cal R}:\Big[\zeta(x) \mu(\vtd) \big(\nabla \vud + (\nabla \vud)^T\big)\Big] \\
\displaystyle -\zeta(x) \mu(\vtd){\cal R}:\big[\nabla \vud+ (\nabla \vud)^T\big]\Big)},
\end{array}
\end{equation}
where ${\cal R}:\tn{A}:=\sum_{i,j=1}^3(\nabla\otimes\nabla\Delta^{-1})_{ij}\tn{A}_{ij}$
for a tensor valued function $\tn{A}$.
Applying Lemma \ref{lem_com2} 
to the second term in \eqref{vf3} we finally obtain the {\em effective viscous flux identity}:
\begin{equation} \label{vf_final}
\begin{array}{c}
\displaystyle
\overline{p(\vr,\vt) T_k(\vr)} - \Big(\frac 43 \mu(\vt) + \xi(\vt)\Big) \overline{T_k(\vr) \Div \,\vu} \\
\displaystyle =  \overline{p(\vr,\vt)} \,\, \overline{T_k(\vr)} - \Big(\frac 43 \mu(\vt) + \xi(\vt)\Big) \overline{T_k(\vr)} \Div \,\vu.
\end{array}
\end{equation}

\smallskip

\noindent {\bf Step 2. Renormalized continuity equation.}
In the next step we verify that $(\vr,\vu)$ satisfies the renormalized continuity equation. 
For this purpose we introduce the {\em oscillations defect measure}:
\begin{equation} \label{4.11a}
\mbox{{\bf osc}}_{\mathbf q} [\vrd\to\vr](Q) = \sup_{k>1} \Big(\limsup_{\delta \to 0^+} \int_Q |T_k(\vrd)-T_k(\vr)|^q \dx\Big).
\end{equation}
Applying \eqref{vf_final} we show (\cite[Lemma 4.5]{NoPo2}): 
\begin{lem} \label{l 7.4}
Let $(\vrd,\vud,\vtd)$ satisfy 
\begin{equation} \label{conv}
\begin{array}{c}
\vrd \rightharpoonup \vr \qquad \mbox{ in } L^1(\Omega), \\
\vud \rightharpoonup \vu \qquad \mbox{ in } L^r(\Omega), \\
\nabla \vud \rightharpoonup \nabla \vu \qquad \mbox{ in } L^r(\Omega), \quad r>1.
\end{array}
\end{equation}
Assume further that $m>\max\{\frac{2}{3(\gamma-1)},\frac 23\}$. 
Then there exists $q>2$ such that  
\begin{equation} \label{osc}
\mbox{{\bf osc}}_{\mathbf{q}} [\vrd\to\vr](\Omega) < \infty.
\end{equation}
Moreover,
\begin{equation} \label{lim1}
\limsup_{\delta \to 0^+} \intO{ \frac{1}{1+\vt}|T_k(\vrd) -T_k(\vr)|^{\gamma+1}} \leq \intO {\frac{1}{1+\vt} \Big(\overline{p(\vr,\vt)T_k(\vr)}  - \overline{p(\vr,\vt)} \, \, \overline{T_k(\vr)}\Big)}.
\end{equation}
\end{lem}
It is known (see  \cite[Lemma 3.8]{FeNoB}) 
that \eqref{conv} together with \eqref{osc} implies that $(\vr,\vu)$ satisfies the 
renormalized continuity equation.

\smallskip

\noindent {\bf Step 3. Strong convergence of the density.}
As $(\vr,\vu)$ and $(\vrd,\vud)$ satisfy the 
renormalized continuity equation, in particular we have 
$$
\intO{T_k(\vr) \Div \, \vu} = 0, \quad \intO{T_k(\vrd) \Div \, \vud}=0
$$
and the second identity implies
$$
\intO{\overline{T_k(\vr) \Div \, \vu}} = 0.
$$
Therefore using \eqref{vf_final} we get 
\begin{equation} \label{4.11b}
\begin{array}{c} 
\displaystyle \intO {\frac{1}{\frac 43 \mu(\vt) + \xi(\vt)}\Big(\overline{p(\vr,\vt) T_k(\vr)} -
\overline{p(\vr,\vt)}\, \, \overline{
T_k(\vr)}\Big)} \\
\displaystyle = \intO{\big(T_k(\vr) - \overline{T_k(\vr)}\big) \Div \vu} \to_{k \to \infty} 0,
\end{array}
\end{equation}
which together with \eqref{lim1} implies 
$$
\lim_{k \to \infty} \limsup_{\delta \to 0^+} \intO{|T_k(\vrd)
-T_k(\vr)|^q} = 0
$$
with $q$ as in Lemma \ref{l 7.4}. It remains to use the fact that 
$$
\|\vrd-\vr\|_1 \leq \|\vrd-T_k(\vrd)\|_1 + \|T_k(\vrd) -T_k(\vr)\|_1 + \|T_k(\vr) -\vr\|_1,
$$
which yields strong convergence of the density in $L^1$, therefore also in $L^p$
for $1 \leq p < s\gamma$. 

The above strong convergence of the density allows to remove all the bars in 
\eqref{weak_mom_lim}, \eqref{entropy_aprox_lim} and \eqref{8.20}. 
Collecting all the assumptions on $m$ we see that the most 
restrictive constraint is $m>\frac{2}{3(\gamma-1)}$ and for weak solutions
we must take into account $m>{\rm max}\{1,\frac{2\gamma}{3(3\gamma-4)}\}$. 
This completes the proof of Theorem \ref{t1}. 

\smallskip

{\bf Acknowledgement:}  The work of the first author (TP) was supported by the Polish NCN grant No. UMO-2014/14/M/ST1/00108. The second author (MP) was supported by the Czech Science Foundation (grant no. 16-03230S).

\end{document}